\documentclass[11pt, reqno, a4paper]{amsart}

\usepackage{amssymb, amsmath, amscd}
\usepackage{amsthm}

\usepackage{verbatim}

\usepackage{color}

\usepackage{etex}
\usepackage{dsfont}
\usepackage[matrix, arrow, curve, frame, arc]{xy}
\usepackage{pgf}
\usepackage{tikz-cd}
\usetikzlibrary{arrows,shapes,automata,backgrounds,petri,calc}
\usepackage[square,sort,comma,numbers]{natbib}
 \usepackage{url}
 \usepackage{hyperref}
 \usepackage[shortlabels]{enumitem} 

\newcommand{\tikzAngleOfLine}{\tikz@AngleOfLine}
\def\tikz@AngleOfLine(#1)(#2)#3{%
\pgfmathanglebetweenpoints{%
\pgfpointanchor{#1}{center}}{%
\pgfpointanchor{#2}{center}}
\pgfmathsetmacro{#3}{\pgfmathresult}%
}

\newcommand{\bS}{\mathbb{S}}

\newcommand{\bN}{\mathbb{N}}
\newcommand{\bM}{\mathbb{M}} 
\newcommand{\bE}{\mathbb{E}} 
\newcommand{\bD}{\mathbb{D}}

\newcommand{\bfQ}{{\bf Q}}
\newcommand{\bbQ}{\mathbb{Q}} 

\newcommand{\bZ}{\mathbb{Z}}

\newcommand{\wt}{\widetilde}
\newcommand{\ba}{\bar{\alpha}}
\newcommand{\La}{\Lambda}

\newcommand{\ve}{\varepsilon}

\newcommand{\ul}{\underline}
\newcommand{\rad}{\operatorname{rad}}
\newcommand{\soc}{\operatorname{soc}}
\renewcommand{\mod}{\operatorname{mod}}
\newcommand{\Hom}{\operatorname{Hom}} 
\newcommand{\Img}{\operatorname{Im}} 
\newcommand{\Adj}{\operatorname{Ad}}
\newcommand{\Arr}{\operatorname{Ar}}

\begin{document}

\newtheorem{defi}{Definition}[section]
\newtheorem{rem}[defi]{Remark}
\newtheorem{prop}[defi]{Proposition}
\newtheorem{ques}[defi]{Question}
\newtheorem{lemma}[defi]{Lemma}
\newtheorem{cor}[defi]{Corollary}
\newtheorem{thm}[defi]{Theorem}
\newtheorem{expl}[defi]{Example} 
\newtheorem*{mthm}{Main Theorem} 
\newtheorem*{remk}{Remark}

\parindent0pt

\title[Periodicity Shadows I]{Periodicity shadows I: A new approach to combinatorics of periodic algebras$^*$  
\footnote{\tiny $^*$ This paper was financed from the grant no. 2023/51/D/ST1/01214 of the Polish National Science 
Centre}} 

\author[A. Skowyrski]{Adam Skowyrski}
\address[Adam Skowyrski]{Faculty of Mathematics and Computer Science, 
Nicolaus Copernicus University, Chopina 12/18, 87-100 Torun, Poland}
\email{skowyr@mat.umk.pl} 


\begin{abstract} This article is devoted to introduce a new notion of periodicity shadow, which appeared naturally 
in the study of combinatorics of tame symmetric algebras of period four, or more generally, algebras of generalized 
quaternion type \cite{AGQT}. For any such an algebra $\La$, we consider its shadow $\bS_\La$, which is the (signed) 
adjacency matrix of the Gabriel quiver of $\La$. Studying properties of shadows $\bS_\La$ leads us to the definition 
of the periodicity shadow, which is basically, a skew-symmetric integer matrix satisfying certain set of conditions 
motivated by the properties of shadows $\bS_\La$. This turned out to be a very useful tool in describing the 
combinatorics of Gabriel quivers of algebras of generalized quaternion type, not only for algebras with small Gabriel 
quivers (i.e. up to $6$ vertices), which it was originally desined for. In this paper, we introduce and briefly discuss 
this notion and present one of its theoretical applications, which shows how significant it is. Namely, the main result 
of this paper describes the global shape of the Gabriel quivers of algebras of generalized quaternion type, as quivers 
obtained from some basic shadows by attaching $2$-cycles, and moreover, postion of the $2$-cycles is restricted by precise 
rules (see the Main Theorem). Computational aspects are reported in the second part \cite{BSapx}. 
\end{abstract}

\maketitle 

\section{Introduction}\label{sec:1} 

Throughout this paper, by an algebra we mean a basic finite-dimensional associative $K$-algebra over an 
algebraically closed field $K$. For an algebra $\La$, we denote by $\mod \La$ the category of all finitely 
generated (right) $\La$-modules, and by $D$ the standard duality $D=\Hom_K(-,K)$ on $\mod \La$. Recall that 
for a quiver $Q$, the {\it path algebra} $KQ$ of $Q$ is a vector space over $K$ with basis given by all 
paths of length $\geqslant 0$ in $Q$, and multiplication given by concatenation of paths (see \cite{ASS,Sk-Y}). An 
ideal $I$ of $KQ$ is called {\it admissible} if $R_Q^m\subseteq I \subseteq R_Q$, for some $m\geqslant 2$, where $R_Q$ 
is the ideal of $KQ$ generated by all paths of length $\geqslant 1$ ($R_Q$ is the Jacobson radical of $KQ$). 
By the well known results of Gabriel \cite{Ga1,Ga2}, every (basic) algebra $\La$ over $K$ admits an isomorphism 
$\La\cong KQ/I$, where $Q$ is a quiver and $I$ an admissible ideal of $KQ$. Every such an isomorphism is said to 
be a {\it presentation} of $\La$ and $(Q,I)$ is sometimes called a {\it bound quiver}. The quiver $Q$ does not depend 
on the presentation (up to isomorphism), it is called the {\it Gabriel quiver} of $\La$, and denoted by $Q_\La$. Instead 
of just presentation, it is  sometimes said that the algebra has a presentation {\it by quiver and relations}, because 
an ideal $I$ is admissible if and only if it is generated by a finite number of so called relations (see Section 
\ref{sec:2}). In other words, every algebra is determined by a quiver $Q=Q_\La$ and a finite set of relations, which 
are given by paths in $Q$ (and some scalars from $K$). \medskip 

This article can be viewed as a part of algebraic combinatorics with representation-theoretical flavour. 
Actually, we will use combinatorial notion (the periodicity shadow) reflecting the structure of module categories 
to derive information on the structure of Gabriel quivers of algebras in a particular class we are interested in. 
The algebras we want to study are the tame symmetric algebras of period $4$, or more generally, so called algebras of 
generalized quaternion type, which form important subclasses of the class of selfinjective algebras (see discussion below). \medskip 

From the remarkable Tame and Wild Theorem of Drozd (see \cite{Dro,CB}) the class of algebras over $K$ may be divided 
into two disjoint classes. The first class consists of the {\it tame algebras} for which the indecomposable 
modules occur in each dimension $d$ in a finite number of discrete and a finite number of one-parameter families. 
The second class is formed by the {\it wild algebras}, whose module category contains module categories of all 
algebras over $K$, and informally speaking, there is no hope to classify indecomposable modules over wild algebras. 
The distinction between tame and wild algebras is based on the 'possibility of classification' of indecomposable 
modules, however we will not use tameness in this context. In this paper, we basically exploit the fact that 
tame algebras seem to be more accessible in terms of classifying algebraic structures, in particular, their Gabriel 
quivers cannot be too big (see also Section \ref{sec:2}). \medskip 

A prominent role in the representation theory of algebras is played by the selfinjective algebras. Recall 
that $\La$ is said to be {\it selfinjective} if and only if $\La$ is injective as a $\La$-module, or 
equivalently, all projective modules in $\mod \La$ are injective. Note that selfinjective algebras are 
exactly the basic Frobenius algebras, so any such an algebra admits a non-degenerate associative $K$-bilinear 
form $(-,-):\La\times \La\to K$, by a classical result of Brauer, Nesbitt and Nakayama \cite[see Theorem IV.2.1]{Sk-Y}. 
If this form is additionally symmetric, i.e. $(a,b)=(b,a)$, for all $a,b\in \La$, then the algebra is called 
{\it symmetric}, and the symmetric algebras form a well known class of selfinjective algebras. There are many classical 
examples of symmetric algebras, for instance, blocks of finite-dimensional group algebras \cite{E1} or Hecke algebras 
associated to Coxeter groups \cite{ad1}. Any algebra $\La$ is a quotient of its trivial extension $T(\La)$, which 
is a symmetric algebra. \smallskip 

For a module $M$ in $\mod\La$, its {\it syzygy} is a kernel $\Omega(M)$ of a projective cover of $M$ in $\mod\La$. 
A module $M$ in $\mod\La$ is called {\it periodic} if $\Omega^d(M)\simeq M$, for some $n\geqslant 1$, and the smallest 
such a number is the {\it period} of $M$. The notion of {\it inverse syzygy} $\Omega^{-1}(M)$ for a module $M$ in $\mod \La$ 
is defined dually using injective envelopes. For $\La$ selfinjective, the syzygy operator induces an equivalence of 
the stable module category $\ul{\mod}\La$, and its inverse is the shift of a triangulated structure on $\ul{\mod}\La$ 
\cite{Happ}. \smallskip 

An important class of selﬁnjective algebras consists of the {\it periodic} algebras $\La$, for which $\La$ is a periodic 
module in $\mod \La^e$, where $\La^e = \La^{op} \otimes_K \La$ is the enveloping algebra of $\La$ (this is equivalent 
to say that $\La$ is periodic as a $\La$-$\La$-bimodule). Every periodic algebra $\La$ has periodic module category,
that is, all (nonprojective) modules in $\mod \La$ are periodic with period dividing the period of $\La$. Note that 
periodicity of simple modules in $\mod \La$ implies that $\Lambda$ is selfinjective, due to \cite[Theorem 1.4]{GSS}. 
Finding and possibly classifying periodic algebras is an important problem, as they appear in many places, revealing 
connections with group theory, topology, singularity theory, cluster algebras and algebraic combinatorics (see the 
survey \cite{ESk08} and the introduction of \cite{WSA1} and \cite{VM}). \medskip 

In this paper we develop combinatorial tools, which allow to describe the general structure of Gabriel 
quivers of all tame symmetric algebras with simple modules periodic of period $4$ (in particular, all 
periodic algebras of period $4$). Combining matrix equations and some covering techniques, we obtain quite strong 
restrictions on the Gabriel quivers, including the knowledge on the position of $2$-cycles and the local 
structure around them (see the Main Theorem). \smallskip 

Restriction only to period $4$ is motivated by a conjecture that all tame symmetric periodic algebras of non-polynomial 
growth have period $4$ \cite[see Problem]{WSA1}. Any tame symmetric periodic algebra will be called TSP4, for short. Such 
algebras arised naturally in the study of blocks of group algebras with generalized quaternion defect groups. Using 
known properties of these blocks, Erdmann introduced and investigated in \cite{nosim,QT1,QT2,E1} the algebras of 
{\it quaternion type}, as the (indecomposable) representation-infinite tame symmetric algebras with non-singular Cartan 
matrix for which every indecomposable non-projective module is periodic of period dividing $4$. The most recent extension 
concerns so called algebras of {\it generalized quaternion type} \cite{AGQT}, or simply GQT algebras, which are by deﬁnition, 
the (indecomposable) tame symmetric algebras of infinite representation type with all simple modules of period $4$. Note that
any algebra of quaternion type, or more generally, representation-infinite TSP4 algebra, is a GQT algebra, since existence of 
a simple module of period $2$ implies that the algebra is of finite representation type (see \cite{note}). \smallskip 

At this point, we also mention that the techniques introduced in this article seem to have strong consequences. Namely, 
besides the Main Theorem, we obtain as an immediate corollary (see Corollary \ref{coro1} and discussion after) a result 
generalizing nontrivial research presented in \cite{QT1,QT2}. \medskip 

The key notion of this paper, i.e. {\it periodicity shadow}, is motivated by a certain property 
of the adjacency matrix associated to the Gabriel quiver of a GQT-algebra. An unexpected 
discovery showed that for any GQT algebra $\La=KQ/I$ (or more generally, any symmetric algebra with 
simples of period $4$), its Cartan matrix $C=C_\La$ satisfies the following equation 
$$\Adj_Q\cdot C=0,$$ 
where $\Adj_Q$ is so called {\it signed adjacency matrix} of the Gabriel quiver $Q=Q_\La$ of $\La$, 
measuring the differences between the number of arrows $i\to j$ and $j\to i$ in $Q$ (see Section \ref{sec:2}, 
Theorem \ref{thm:motivation}). This property serves as a tool to classify all possible matrices $\Adj_Q$, 
and hence, also Gabriel quivers of considered algebras. For any GQT algebra $\Lambda$, we denote by 
$\bS_\Lambda$ the adjacency matrix of its Gabriel quiver $\Adj_{Q_\Lambda}$, and call it the {\it shadow} 
(of algebra $\La$). \smallskip 

The idea behind introducing the concept of periodicity shadow is to catch the combinatorics of Gabriel 
quivers of GQT algebras reflected in the properties of their associated adjacency matrices. Imposing tameness 
assumption, we also obtain strong restrictions of the coefficients of associated shadow, which cannot exceed 
$\pm 2$. Motivated by the above equation, we study a few necessary conditions on matrix $A$, to be a shadow of a 
GQT algebra, and call it a {\it (tame) periodicity shadow}, if all these conditions are satisfied. In other words, 
we study particular superset of the set of all shadows $\bS_\La$ of GQT algebras $\La$. \smallskip 

The main object of our interest is the structure of GQT algebras and related subclasses, though we will 
not deal with algebras directly. Crucial property of these algebras is periodicity of simples, which justifies the name 'shadow', 
choosed to stress that we are dealing with something that reflects periodicity (of simples) at the combinatorial level of 
the associated skew-symmetric matrices. For more details we refer to Section \ref{sec:3}. \medskip 

Computing all periodicity shadows is a different issue, which is discussed in details in part II \cite{BSapx}. 
We only mention that even for small number $n$ of vertices, say $n\leqslant 6$, it is a challenge, whereas for bigger 
$n$, it quickly becomes intractable in practice. All we know for $n\geqslant 8$ is that there exists a set of tame 
periodicity shadows, and there is an algorithm (based on recursive generating) which can compute it, at least in theory. 
Classifying Gabriel quivers of GQT algebras via shadows is possible for $n\in\{3,\dots,6\}$, which are the most interesting 
cases. We will not elaborate more on this topic, which is a future research project, related to classification of all GQT 
algebras with small quivers. Our goal here is to focus on theoretical applications of the notion of periodicity shadow. 
The first instance is the following theorem, which is the main result of this paper. \medskip 

\begin{mthm} Let $n\geqslant 1$ be a natural number. Then there is a finite set $\bS(n)$ of tame periodicity shadows $n\times n$ 
such that every GQT algebra $\Lambda$ with Gabriel quiver $Q=Q_\La$ having $n$ vertices satisfies the following conditions. 
\begin{enumerate} 
\item[(a)] A minimal subquiver $Q^\times$ of $Q$ without loops and $2$-cycles has $\Adj_{Q^\times}\in\bS(n)$ (up to relabelling 
of vertices). 
\item[(b)] There are $2$-cycles $\xymatrix{a_i\ar@<-0.4ex>[r] & b_i \ar@<-0.4ex>[l]}$ in $Q$, $1\geqslant i \geqslant m$, such 
that $\{a_i,b_i\}$ and $\{a_j,b_j\}$ are disjoint for all $i\neq j$, and $Q$ is obtained from $Q^\times$ by attaching these 
$2$-cycles (and possibly loops). 
\end{enumerate} 
Moreover, any of the $2$-cycles in (b) is contained in one of the following blocks in $Q$: 
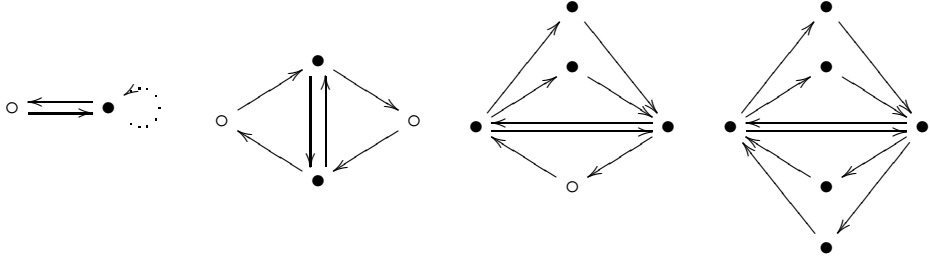
$$\xymatrix@R=0.4cm{\\ \\ \circ \ar@<-0.45ex>[r] & \bullet \ar@<-0.45ex>[l] \ar@(rd,ru)@{..>}[]& }
\xymatrix@R0.4cm{\\ & \bullet \ar[rd] \ar@<-0.1cm>[dd] & \\ 
\circ \ar[ru]   && \circ \ar[ld]\\ 
& \bullet \ar@<-0.1cm>[uu] \ar[lu] & } \quad 
\xymatrix@R=0.4cm{ &\bullet \ar[rdd] & \\ & \bullet\ar[rd]& \\ 
\bullet \ar[ru] \ar[ruu] \ar@<-0.35ex>[rr]&& \bullet \ar[ld] \ar@<-0.35ex>[ll]\\ & \circ \ar[lu] & } \quad 
\xymatrix@R=0.4cm{& \bullet \ar[rdd] & \\ & \bullet \ar[rd] & \\ 
\bullet \ar[ru]\ar[ruu] \ar@<-0.35ex>[rr] && \ar[ld]\ar[ldd] \ar@<-0.35ex>[ll] \bullet \\ 
& \bullet \ar[lu] & \\ & \ar[luu] \bullet & }$$ 
\end{mthm} 

In the sequel, we will sometimes call the above theorem, the Reconstruction Theorem, to stress that it describes how to 
reconstruct the Gabriel quiver of an algebra from its shadow by adding $2$-cycles (and possibly loops). \smallskip 

We expect to use shadows as a conceptual tool to provide some further restrictions on the general shape of $Q$, the above theorem 
is the first one. This is related to a recent conjecture that the class of so called {\it generalized weighted surface algebras} 
\cite{GWSA} exhaust all GQT algebras. Without going into technical details, let us only mention that the generalized weighted 
surface algebras form the largest known class of GQT algebras which arised as a natural generalization (related to mutations) 
of the {\it virtual mutations} \cite{VM} and {\it weighted surface algebras} \cite{WSA1,WSA2,WSA3}. The weighted surface 
algebras originated from the cluster theory and they have a combinatorial description (by quivers and relations) based on 
combinatorics of a surface with triangulation. Similar description exists for generalized weighted surface algebras, and 
roughly speaking, they are given as quotients $KQ/I$, where $Q$ is a glueing of a finite number of blocks of types I-V, but 
$I$ may not be admissible \cite[see Introduction]{GWSA}. Analyzing what can be the position of $2$-cycles in Gabriel quivers 
of generalized weighted surface algebras, one can see that these are exactly as described in the Main Theorem, which confirm the 
conjecture, at least at the level of $2$-cycles. The next natural step (we belive within reach) is to use the notion of 
shadow to prove that Gabriel quivers of GQT algebras admit (modulo $2$-cycles under control) the same form as Gabriel quivers 
of generalized weighted surface algebras, i.e. they are glueings of blocks of types I-V (without arrows in $I$). This would be 
a big contribution to the solution of the conjecture in general, nonetheless, our main result is a nontrivial step towards. \medskip 

The paper is organized as follows. In Section \ref{sec:2} we give a brief introduction to basic notion used later, 
and discuss motivation for this study. Next short section serves a few necessary results concerning combinatorics 
of Gabriel quivers of tame symmetric algebras of period $4$ needed further. Section \ref{sec:3} is devoted to 
introduce the main notion of periodicity shadow and present some basic facts about it. We finish with Section \ref{sec:4}, 
where we provide the proof of the Main Theorem. \medskip 

For the necessary background on the representation theory of algebras, we refer the reader to the books \cite{ASS,Sk-Y}.

\bigskip

\section{Preliminaries and motivation}\label{sec:2} 

By a quiver we mean a quadruple $Q=(Q_0,Q_1,s,t)$, where $Q_0$ is a finite set of vertices, $Q_1$ a finite set of arrows 
and $s,t:Q_1\to Q_0$ functions assigning to every arrow $\alpha$ its source $s(\alpha)$ and its target $t(\alpha)$. 
For a quiver $Q$, we denote by $KQ$ the {\it path algebra} of $Q$, whose $K$-basis is given by all paths of length 
$\geqslant 0$ in $Q$. Recall that the Jacobson radical of $KQ$ is the ideal $R_Q$ of $Q$ generated by all paths of 
length $\geqslant 1$, and an ideal $I$ of $KQ$ is called {\it admissible}, provided that $R_Q^m\subseteq I \subseteq R_Q^2$, 
for some $m\geqslant 2$. Moreover, the trivial paths $\ve_i$ (of length $0$) at vertices $i\in Q_0$ form a complete 
set of (pairwise orthogonal) primitive idempotents of $KQ$ with $\sum_{i\in Q_0}\ve_i$ the unit of $KQ$. \medskip 

If $Q$ is a quiver and $I$ an admissible ideal $I$ of $KQ$, then $(Q,I)$ is said to be a {\it bound quiver}, 
and the associated algebra $KQ/I$ is called a {\it bound quiver algebra}. It is well-known that any algebra 
$\Lambda$ over an algebraically closed field admits a bound quiver algebra form, and by a presentation of $\Lambda$ 
we mean particular isomorphism $\Lambda\cong KQ/I$, for some bound quiver $(Q,I)$. In this case, the cosets 
$e_i=\varepsilon_i+I\in \La$ form a complete set of primitive orthogonal idempotents of $\La$ and $\sum_{i\in Q_0}e_i$ 
is the unit of $\Lambda$. \smallskip 

We recall that for any algebra $\La$ and a complete set of primitive orthogonal idempotents $e_i$, $i\in\{1,\dots,n\}$, 
with $\sum e_i=1_\La$, the {\it Gabriel quiver} of $\Lambda$ (sometimes also called the {\it ordinary quiver}) is the 
quiver $Q_\La$ with vertex set $Q_0=\{1,\dots,n\}$ and the arrows $i\to j$ are in bijective correspondence with the 
vectors in a basis of the $K$-vector space $e_i(J/J^2)e_j$, where $J$ is the Jacobson radical $J=J_\La=R_Q+I$ of 
$\Lambda$. This quiver is uniquely determined by the algebra, and does not depend on the choice of a complete set of 
primitive othogonal idempotents \cite[see II.3]{ASS}. Moreover, for any presentation $\Lambda=KQ/I$, the Gabriel 
quiver of $\La$ is $Q_\La=Q$, and $Q_\Lambda$ is connected if and only if $\La$ is indecomposable as an algebra. \medskip 

A {\it relation} in the path algebra $KQ$ is any $K$-linear combination of the form 
$$\sum_{i=1}^r \lambda_i w_i,\leqno{(1)}$$ 
where all $\lambda_i\in K$ are non-zero and $w_i$ are pairwise different paths of of length $\geqslant 2$ with common 
source and target. It is known that an ideal $I$ of $KQ$ is admissible if and only if $I$ is generated by a finite number 
of relations $\rho_1,\dots,\rho_m$. Moreover, we may choose such relations $\rho_1,\dots,\rho_m$ to be minimal (i.e. each 
$\rho_i$ is not a combination of relations from $I$). For a bound quiver algebra $A=KQ/I$, given the set of (minimal) 
relations $\rho_1,\dots,\rho_m$ generating $I$, we have the (minimal) equalities $\rho_1=0,\dots,\rho_m=0$, called 
{\it minimal relations}. \smallskip 

For a relation of the form $(1)$ and a path $w$ in $Q$, we write $w\prec \rho$, if $w$ is one of the summands 
of $\rho$, i.e. $w=w_k$, for some $k\in\{1,\dots,r\}$. Moreover, if $w$ is a path in $Q$, we will use notation 
$w\prec I$, if $w\prec \rho_j$, for $j\in\{1,\dots,m\}$ and some choice of minimal relations $\rho_1,\dots,\rho_m$ 
generating $I$. Changing the set of minimal generators of $I$, the resulting set of paths that are $\prec I$ may 
also change. When we write $w\prec I$, we always have a fixed set of minimal relations generating $I$. \medskip

Let $\Lambda$ be an indecomposable algebra with given presentation $\Lambda=KQ/I$. Then modules $P_i=e_i\La$, 
for $i\in Q_0$, form a complete set of all pairwise non-isomorphic indecomposable projective modules in $\mod\La$, 
and modules $I_i=D(\La e_i)$, for $i\in Q_0$, form a complete set of all pairwise non-isomorphic indecomposable 
injective modules in $\mod\La$. We denote by $S_i$, for $i\in Q_0$, the associated simple module 
$S_i=P_i/\rad P_i\cong \soc I_i$. To avoid details with relabeling, we assume that the given algebra $\La$ 
has a presentation $\La=KQ/I$, where $Q_0$ has a fixed ordering $Q_0=\{1,\dots,n\}$. \smallskip 

Let us note here a few consequences of the tameness assumption. First of all, note that if $\La$ is a tame 
algebra, then its Gabriel quiver $Q$ has at most two arrows from $i$ to $j$, for any $i,j\in Q_0$. Indeed, 
if this was not the case, then $Q$ admits a subquiver of the form 
$$\Delta=K_3= \xymatrix@C=0.5cm{\circ\ar@<+0.15cm>[rr]\ar@<-0.15cm>[rr]\ar[rr] && \circ}$$ 
and the path algebra $K\Delta$ is a quotient of $\La$. But $K\Delta$ is a wild algebra, so $\La$ is 
also wild, a contradiction. Similarily, $Q$ does not contain a subquiver of the form 
$$K_2^+=\xymatrix@C=0.3cm{\circ&&\ar[ll]\circ\ar@<+0.4ex>[rr]\ar@<-0.4ex>[rr] && \circ} \ \ \mbox{ or } \ \ 
K_2^-=\xymatrix@C=0.3cm{\circ\ar[rr]&& \circ && \ar@<+0.4ex>[ll]\ar@<-0.4ex>[ll] \circ} $$  
because its path algebra is a wild algebra, hence cannot be a quotient of $\La$. In a similar way, any quiver $Q$ 
of a tame algebra cannot contain a star 
$$S_5^+= \quad \xymatrix@R=0.45cm{\circ & \circ & \circ  \\ 
\circ  & \circ \ar@{->}[r] \ar@{->}[l] \ar@{->}[u] \ar@{->}[ru] \ar@{->}[lu] & \circ }\qquad\mbox{ or }\qquad 
S_5^-= \quad \xymatrix@R=0.45cm{\circ & \circ & \circ  \\ 
\circ  & \circ \ar@{<-}[r] \ar@{<-}[l] \ar@{<-}[u] \ar@{<-}[ru] \ar@{<-}[lu] & \circ }$$ 
representing a source or sink of five arrows. \smallskip 

Moreover, for any two vertices $i,j$ connected by double arrows $\alpha,\ba$, there are no loops at $i,j$. Indeed, 
otherwise there exists a quotient algebra $C$ of $B=\La/J^3$ and a Galois covering $\wt{C}\to C$ such that $\wt{C}$ 
contains a full convex subcategory of one of the forms 
$$\xymatrix@C=0.6cm{\circ &\ar@<+0.4ex>[l]^{}\ar@<-0.4ex>[l]_{}\circ\ar[r]&\circ} \qquad \mbox{or} \qquad 
\xymatrix@C=0.6cm{\circ \ar@<+0.4ex>[r]^{}\ar@<-0.4ex>[r]_{}&\circ&\ar[l]\circ} $$ 
which is isomorphic to a wild hereditary algebra $K\Delta$, where $\Delta$ is of type $K_2^\pm$. As a result, using 
\cite[Theorem]{DS2} and \cite[Proposition 2]{DS1} we conclude that $\wt{C}$ and $C$ are wild. Since $C$ is a quotient 
of $B$ and $\Lambda$, we obtain that $\La$ is wild, a contradiction. \smallskip 

Using analogous argumentation, one can prove that for any vertex $i\in Q_0$, there is at most one loop at $i$, 
if $Q$ has more than one vertex. In fact, if $Q$ admits two loops at vertex $i$, then appropriate Galois covering 
contains infinite double-line 
$$\xymatrix@C=0.6cm{\dots \ar@<+0.4ex>[r]\ar@<-0.4ex>[r] & \circ 
\ar@<+0.4ex>[r]\ar@<-0.4ex>[r]&\circ\ar@<+0.4ex>[r]\ar@<-0.4ex>[r] & \dots}$$ 
being a 'resolution' of the two loops in $Q$. It follows that $Q$ has exactly one vertex, because otherwise, the 
above line can be extended to a wild subcategory of type $K_2^\pm$. \smallskip 

Finally, we will say that $\Delta$ is a subquiver {\it of type} $K_2^*$ if $\Delta$ is on of the following two quivers 
$$ \xymatrix@C=0.6cm{\circ \ar@<+0.4ex>[r]^{\ba}\ar@<-0.4ex>[r]_{\alpha}&\circ\ar[r]^{\beta}&\circ} \ \mbox{ or } \ 
\xymatrix@C=0.6cm{\circ \ar[r]^{\beta}&\circ \ar@<+0.4ex>[r]^{\ba}\ar@<-0.4ex>[r]_{\alpha}&\circ}$$
and both $\alpha\beta, \bar{\alpha}\beta\nprec I$ (or both $\beta\alpha,\beta\bar{\alpha}\nprec I$). Note that 
if $\La$ is tame, its Gabriel quiver $Q=Q_\La$ does not contain a subquiver of type $K_2^*$. To see this, let 
$\Delta$ be a subquiver in $Q$ of type $K_2^*$. Then we consider the factor algebra $C=KQ/J_C$ of $B=\La/J^3$, 
where $J_C$ is generated by all paths of length $2$ in $Q$ except $\alpha\beta$ and $\bar{\alpha}\beta$, in 
the first case, and $\beta\alpha,\beta\ba$, in the second. As before, $C$ admits a Galois covering $\wt{C}\to C$ 
containing a full subcategory 
$$\xymatrix@C=0.6cm{\circ \ar@<+0.4ex>[r]^{}\ar@<-0.4ex>[r]_{}&\circ\ar[r]&\circ} \ \mbox{ or } \ 
\xymatrix@C=0.6cm{\circ \ar[r]&\circ \ar@<+0.4ex>[r]^{}\ar@<-0.4ex>[r]_{}&\circ}$$ 
which is isomorphic to a wild hereditary algebra $K\Delta$. \medskip 

The above subquivers can be viewed as the smallest wild 'pieces' that may occur in $Q$. For hereditary algebras of 
Euclidean (or wild) type, we use standard  notation $\wt{\bD}_n,\wt{\bE}_6,\wt{\bE}_7,\wt{E}_8$ 
(or $\wt{\wt{\bD}}_n,\wt{\wt{\bE}}_{6-8}$). \smallskip 

Now, let us only mention one simplification which is also frequently used during the sequel, where 
many proofs contain wordings like "the algebra admits the following wild-one relation algebra ... in covering". 
By this we mean a precise statement in the language of covering theory, namely, an algebra $A$ is said to 
be a {\it subcategory in covering} of $\La$, provided that the algebra $B=\La/J^d$ (for some $d\geqslant 3$) 
has a quotient algebra $C$, which admits a Galois covering $\tilde{C}\to\tilde{C}/G=C$, with a finitely generated 
free group $G$, such that $\tilde{C}$ contains a full subcategory isomorphic to $A$. \smallskip 

At any time when we use this abbreviation, appropriate arguments can be verified, but we do not elaborate on this, to 
avoid making  long proofs even longer. One simple instance of an argumentation of this type was shown above 
in the case of subquivers of type $K_2^*$, where the essential results needed from covering theory are cited. 
For more background on covering techniques we refer to the papers \cite{DS1,DS2}. \bigskip 

All algebras $\La$ are assumed symmetric, hence in particular, we have $P_i=I_i$, for any $i\in Q_0$. We also assume 
$Q$ is connected, i.e. $\La$ is indecomposable as an algebra. For a vertex $i\in Q$, we denote by $p_i$ the dimension 
vector of $P_i$, and by $s_i$ the dimension vector of the simple module $S_i$. Recall that the {\it Cartan matrix} $C_\La$ 
of an algebra $\La$ is defined as $C_\La=[c_{ij}]$, where $c_{ij}=\dim_K e_j\La e_i$, that is, its columns are the dimension 
vectors $p_1,\dots,p_n$ of indecomposable projectives. Note that $C_\Lambda$ is a symmetric matrix with natural 
coefficients, if $\Lambda$ is a symmetric algebra. In fact, this is the case for more general class of algebras, 
called {\it weakly symmetric algebras}. \smallskip 

For $i\in Q_0$, we let $i^-$ be the set of arrows ending at $i$, and $i^+$ the set of arrows starting at $i$. Then $Q$ 
has a subquiver 
$$
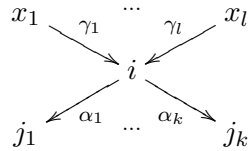
\xymatrix@R=0.3cm{x_1\ar[rd]^{\gamma_1}& ^{\dots} & x_l \ar[ld]_{\gamma_l} \\ 
&i\ar[ld]^{\alpha_1} \ar[rd]_{\alpha_k} & \\ j_1 & ^{\dots} & j_k}$$ 

A vertex with $|i^-|=p$ and $|i^+|=q$ is said to be $(p,q)$-$regular$, or just a $(p,q)$-$vertex$. If $p=q$, then $i$ 
is called $p$-$regular$ (or simply, a $p$-vertex). \smallskip 

Recall that there are natural isomorphisms \cite[see Lemma 4.1]{AGQT}
$$\Omega(S_i)=\rad P_i=\sum_{s=1}^k\alpha_{j_k}\La \quad\mbox{ and }\quad 
\Omega^{-1}(S_i)\cong (\gamma_1,\dots,\gamma_l)\Lambda\subset \bigoplus_{s=1}^l P_{x_s}.$$ 
In particular, it follows that the module $P_i^+=\bigoplus_{s=1}^k P_{j_s}$ is a projective cover of $\Omega(S_i)$, 
whereas $P_i^-=\bigoplus_{s=1}^l P_{x_s}$ is an injective envelope of $\Omega^{-1}(S_i)$. As a result, if $S_i$ is 
a periodic module of period $4$, then $\Omega^2(S_i)\simeq\Omega^{-2}(S_i)$, and hence there is an exact sequence in 
$\mod\La$ of the form 
$$0\to S_i\to P_i \stackrel{d_3}\to  P_i^- \stackrel{d_2}\to  P_i^+ \stackrel{d_1}\to P_i \to S_i\to 0 \leqno{(*)}$$ 
with $\Img d_k\cong\Omega^k(S_i)$, for $k\in\{1,2,3\}$. We denote by $p_i^+$ (respectively, $p_i^+$) the dimension vector 
of $P_i^+$ (respectively, of $P_i^-$). This dimension vector will be denoted by $\hat{p}_i$ (see also Lemma \ref{lem:3.1}). 
\medskip 

From now on, we will assume that all simple modules in $\mod\Lambda$ are periodic of period $4$. In particular, 
using the above exact sequences of projectives, we obtain the following identities:  
$$  p_i^+=p_i^-,\hfill\leqno{(p)}$$ 
for any $i\in Q_0=\{1,\dots,n\}$. These equalities have been extensively used in an ongoing work concerning the structure 
of TSP4 algebras; see \cite{EHS1,EHS2}. The identities $(p)$ lead us to the main object of our study, and the next theorem 
shows how they transform into a surprising matrix identity. \medskip 

Before we present the main theorem motivating our study, let us formulate related definition below. \small 

\begin{defi} Let $Q$ be a quiver on $Q_0=\{1,\dots,n\}$. Its {\it arrow matrix} is the matrix $\Arr_Q=[q_{ij}]\in\bM_n(\bN)$, 
whose coeeficients $q_{ij}$ count the number of arrows $i\to j$ in $Q$, for $i,j\in Q_0$. With this, we define 
so called {\bf signed adjacency matrix} $\Adj_Q$ of $Q$ as follows 
$$\Adj_Q:=\Arr_Q-\Arr_Q^T.$$ 
\end{defi} \medskip 

In other words, $\Adj_Q=[a_{ij}]$ measures the differences $a_{ij}=q_{ij}-q_{ji}$ between the number of arrows 
$i\to j$ and $j\to i$ in $Q$, for any pair of vertices $i,j\in Q_0$. Obviously, for any quiver $Q$, its (signed) 
adjacency matrix $A=\Adj_Q$ is skew-symmetric ($A^T=-A$). In particular, it has always $a_{ii}=0$, so 
one can say it does not detect loops in $Q$. Moreover, for any $2$-cycle in $Q$ (i.e. a pair of arrows $i\to j$ 
and $j\to i$), the signed adjacency matrix does not change when we remove it. The term 'signed adjaceny matrix' 
appeared also in \cite{FST08}, where it was a skew-symmetric matrix derived from an ideal triangulation of a 
bordered surface. We only note that if $Q$ does not contain loops and $2$-cycles, then $\Adj_Q$ is exactly the 
{\it exchange matrix} used in the theory of cluster algebras. \medskip 

Now, the crucial observation is stated as follows. 

\begin{thm}\label{thm:motivation} If $\Lambda=KQ/I$ is an algebra of generalized quaternion type, then its Cartan 
matrix $C=C_\Lambda$ satisfies the following identity: 
$$\Adj_Q\cdot C=0.$$ \end{thm} 

\begin{proof} This is straightforward from the identities: $p_i^-=p_i^+$, for $i\in Q_0$. Indeed, for a fixed vertex 
$i\in Q_0$, we have 
$$p_i^-=p_i^+ \ \Leftrightarrow \ \sum_{\alpha\in i^-}p_{s(\alpha)}=\sum_{\alpha\in i^+}p_{t(\alpha)} \ 
\Leftrightarrow \ \sum_{j\to i}p_{j}=\sum_{i\to j}p_j,$$ 
where the last sum is taken over all arrows ending at $i$ (respectively, starting at $i$). Note that the same $j$ 
can appear multiple times, if there is more one arrow $j\to i$ (or $i\to j$). Comparing both sides at $k$-th 
row, we obtain 
$$\sum_{j\to i}c_{kj}=\sum_{i\to j}c_{kj},\hfill\leqno{(c)}$$ 
for any $k\in\{1,\dots,n\}$. Consequently, if $x=[c_{k1} \ \dots c_{kn}]$ is the $i$-th row of the Cartan matrix 
$C=C_\Lambda$, then $(c)$ is equivalent to $a_i\cdot x^T=0$, where 
$$a_i=\sum_{i\to j}e_j-\sum_{j\to i}e_j$$ 
(here $e_j=s_j^T$ denotes the $j$-th vector $[0 \ 0 \dots \ 1 \ \dots 0]$ of the standard basis of $\mathbb{Z}^n$). 
Note that $C$ is symmetric, so $x^T=p_k$, and moreover, $a_i$ is by definition the $i$-th row of the adjacency 
matrix $\Adj_Q$. As a result, we have $\Adj_Q\cdot p_k=0$, for each $k\in\{1,\dots,n\}$, and hence, we get $\Adj_Q\cdot C=0$, 
as required. \end{proof} 

\medskip 

Main idea of this paper is to consider the above identity as a matrix equation 
$$\leqno{(\Delta)} \qquad\qquad\qquad\qquad\qquad\qquad\qquad\qquad\qquad AC=0$$ 
with given $A=\Adj_Q$ and $C$ treated as an unknown (matrix) variable. We aim to use the equation as a tool 
for classifying Gabriel quivers (and possibly Cartan matrices) of tame symmetric periodic algebras of period 4, or 
more generally, algebras of generalized quaternion type. \medskip 

One question is to find all symmetric matrices $C$ with natural coefficients, if $A$ is a given skew-symmetric matrix 
(then clearly, $A$ is of the form $A=\Adj_Q$, for some quiver $Q$). This is more or less, a question of finding natural 
solutions of a system of linear diophantine equations. Existence of $C$ in general is not yet explained in terms of $A$, 
but for preset purposes it is not needed. \smallskip 

More elaborate problem is to answer what are the conditions on a skew-symmetric matrix $A$, under which ($\Delta$) 
has solutions being Cartan matrices of (tame) symmetric periodic algebras of period $4$ (or GQT algebras in general)? 
Related question is to decide whether a given solution of ($\Delta$) is Cartan matrix of a GQT algebra. \medskip 

By definition, each algebra of generalized quaternion type $\Lambda=KQ/I$ induces a solution $C=C^T$ of ($\Delta$), 
where $C$ is the Cartan matrix of $\Lambda$ and $A$ is its shadow $A=\Adj_Q$. In other words, solutions of ($\Delta$) 
contain all information on Cartan matrices of GQT-algebras (and additionally, on adjacency matrices of their Gabriel 
quivers). Even more, we do not need to restrict ourselves only to infinite (or tame) representation type, and we can 
consider shadows $\Adj_{Q_\La}$, for any symmetric algebra $\Lambda$ with all simple modules periodic of 
period $4$ (this class includes all tame symmetric periodic algebras of period $4$, as well as all GQT algebras). 
It is natural to expect that studying equation $(\Delta)$ will give us substantial insights into the structure 
of GQT algebras. \medskip 

In the remaining part of this section we introduce some notation useful in describing quivers $Q$ with 
given adjacency matrix $\Adj_Q=A$. Roughly speaking, we will present a convinient way to separate the 
essential part of the quiver from $2$-cycles. However, we discuss here only the general shape of $Q$. More 
details will arive in Sections \ref{sec:3} and \ref{sec:4}. Actually, in Section \ref{sec:4}, we will prove 
that the position of $2$-cycles in $Q=Q_\Lambda$ is rather restricted (see also the Main Theorem). \medskip 

Let $Q=(Q_0,Q_1)$ be a quiver. We will call $Q$ a {\it reduced} quiver, if it does not contain loops nor $2$-cycles, 
i.e. pairs of arrows $\xymatrix{i\ar@<-0.4ex>[r] & j \ar@<-0.4ex>[l]}$ in $Q$. For any quiver, there is a 
unique reduced subquiver $Q^\times \subseteq Q$ (up a choice of $2$-cycles), and $Q^\times$ is obtained from $Q$ by 
deleting all loops and $2$-cycles. Obviously, $Q$ is reduced if and only if $Q=Q^\times$ and $Q^\times$ is always 
reduced. \smallskip 

There is an equivalence relation $\equiv$ on the set of all quivers, for which $Q \equiv Q'$ if and only if 
$Q$ and $Q'$ differ only by loops and $2$-cycles (equivalently, $Q$ is obtained from $Q'$ by adding or deleting 
loops and $2$-cycles). In this setup, $Q^\times$ is the smallest quiver (with respect to number of arrows) in the 
equivalence class of $Q$ modulo $\equiv$. \smallskip

For two quivers $Q,Q'$ living on a common set of vertices $X=Q_0=Q_0'$, we define their disjoint union as follows  
$$Q\sqcup Q'=(X,Q_1\sqcup Q_1'),$$
where $Q_1\sqcup Q_1'$ denotes the usual disjoint union of sets (of arrows), and the functions $s,t$ on the disjoint 
sum are easily induced from the functions in $Q_1$ and $Q_1'$. \medskip

Now, recall that every skew-symmetric matrix $A=[a_{ij}]\in\bM_n(\bZ)$ represents an equivalence class under the 
relation $\equiv$. Namely, we have a unique decomposition $A=A^++A^-$, where $A^+\geqslant 0$ and $A^-\leqslant 0$ 
(that is, all entries are $\geqslant 0$ or $\leqslant 0$, respectively). Then, there is a unique reduced quiver 
$\bfQ_A$ with $\Arr_{\bfQ_A}=A^+$, i.e. the number of arrows $i\to j$ in $\bfQ_A$ is $a_{ij}$, if $a_{ij}\geqslant 0$, 
and $0$, otherwise. Clearly, the adjacency matrix of $\bfQ_A$ is $\Adj_{\bfQ_A}=A$, and $Q\equiv Q'$ if and only if 
$\Adj_Q=\Adj_{Q'}$. In particular, we conclude that every $Q$ with $\Adj_Q=A$ is obtained from $\bfQ_A=Q^\times$ 
by adding $2$-cycles, and possibly loops. \smallskip 

We would like to separate the case of loops, so we consider also the {\it loop-free part} of $Q$, which is 
the subquiver $Q^\circ$ of $Q$, obtained from $Q$ by deleting all loops. \medskip 

By a {\it graph}, we mean any triple $(G_0,G_1,v)$, where $G_0$ and $G_1$ are the sets of vertices and edges, 
respectively, and $e:G_1\to P_2(G_0)$ is a function attaching to each edge $e\in G_1$ the set $v(e)\in P_2(G_0)$ 
of its {\it ends}, where $P_2(X)$ is a set of all subsets $S\subseteq X$ with exactly $|S|=2$ elements. This 
allows multpile edges (just for generality), but no loops in $G$. We will usually write $(G_0,G_1)$ (or just $G$), 
omitting the function $v$. \medskip 

The graph is called {\it simple}, if there is at most one edge $e\in G_1$ with $v(e)=\{i,j\}$, for all vertices 
$i\neq j$ in $G_0$. Except one case of small graphs, in this paper we will only work with simple graphs. \medskip 

For any graph $G$, we consider its associated {\it double quiver } $\overline{G}=(Q_0,Q_1)$, which is a quiver 
with the same set $Q_0=G_0$ of vertices as $G$, and arrows in $Q_1$ are induced from edges $e\in G_1$, in such 
a way that each edge with $v(e)=\{i,j\}$ induces a $2$-cycle $\xymatrix{i \ar@<-0.4ex>[r] & \ar@<-0.4ex>[l] j}$ 
in $Q_1$. In other words, $\overline{G}$ is obtained from $G$ by replacing each edge by a $2$-cycle. \medskip 

Summing up, we may conclude that for any skew-symmetric matrix $A$, there exists a reduced quiver 
$\bbQ=\bfQ_A$, such that every quiver $Q$ with $\Adj_Q=A$, is given (up to loops) by 
$$Q^\circ=\bbQ\sqcup\overline{G}.$$ 
for some graph $G$. \medskip 

Finally, let us introduce an auxiliary notion of a {\it block}, which will be helpful in compact statements 
of later results. This notion is a natural generalization of {\it blocks} appearing in \cite{WSA3,VM,GWSA} 
(see also \cite{EHS2}), from which Gabriel quivers of many GQT algebras are glued. Blocks contain {\it black} 
and {\it white} vertices, and the white ones (called sometimes {\it outlets}) are the 'glueing' points with 
the remaining part of the quiver. More specifically, by a {\it block} in $Q$ we mean a subquiver $\Gamma=(\Gamma_0,\Gamma_1)$ 
of $Q$, whose vertex set $\Gamma_0$ is a disjoint union $\Gamma_0=B\cup W$ such that every arrow $\alpha:i\to j$ in $Q$ 
with $i$ or $j$ in $B$ satisfies $\alpha\in\Gamma_1$. \smallskip 

In other words, if $\Gamma$ is a block, then all arrows with source or target in $B$ are in $\Gamma_1$. 
In particular, any other arrow $\alpha:i\to j$ in $Q_1\setminus \Gamma_1$ either connects two vertices 
$i,j\in Q_0\setminus \Gamma_0$ or a vertex $Q_0\setminus \Gamma_0$ with an outlet in $W$. Consequently, $Q$ 
can be viewed as a glueing of $\Gamma$ with the rest part of the quiver via its outlets. In the language of 
associated adjacency matrix, it means that $a_{ij}=a_{ji}=0$, for any $i\notin\Gamma_0$ and $j\in B$, and 
hence, matrix $\Adj_Q$ has the following quasi-block form  
$$\Adj_Q=\left[\begin{array}{cc} \Adj_\Gamma & X \\ -X^T & * \end{array}\right],$$ 
where the first $b=|B|$ rows of $X$ (resp. columns of $-X^T$) corresponding to black vertices are zero. \smallskip 

For example, suppose that $Q$ admits the following block $\Gamma$: 
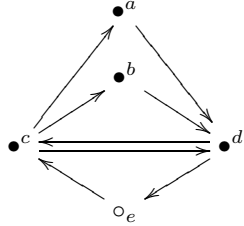
$$\xymatrix@R=0.4cm{ &\bullet^a \ar[rdd] & \\ & \bullet^b\ar[rd]& \\ 
\bullet^c \ar[ru] \ar[ruu] \ar@<-0.4ex>[rr]&& \bullet^d \ar[ld] \ar@<-0.4ex>[ll]\\ & \circ_e \ar[lu] & }$$ 
It follows that all the arrows with source or target in $\{a,b,c,d\}=B$ are visible in the picture, 
and we have exactly one outlet in $W=\{e\}$ (black vertives are usually denoted by $\bullet$, while 
white by $\circ$). In particular, then $a$ and $b$ are $1$-vertices in $Q$, $c$ is a $(2,3)$-vertex, and $d$ is 
a $(3,2)$-vertex in $Q$. Moreover, all other arrows in $Q_1\setminus \Gamma_1$, are either connecting vertices in 
$Q_0\setminus \Gamma_0$, or a vertex in $Q_0\setminus B$ with the unique outlet $e$. As a result, $Q$ is a 
'glueing' of the above block with the rest part of the quiver via the unique outlet $e$. 

\bigskip 

\section{Tame symmetric algebras of period four}\label{sec:2+} 

In this short section we recall the most important properties of tame symmetric algebras of period four 
needed in furter considerations. Some of the basic background was given in Section \ref{sec:2}, mostly 
concerning exact sequences coming from periodic simples. Here we will focus more on the structure of 
their Gabriel quivers and general behaviour of relations. This is based on a recent paper \cite{EHS1}. \medskip 

We start with two observations concerning infinite type \cite[see Lemmas 2.1 and 2.4]{EHS1}. For 
a natural vector $x\in\bN^n$, we denote by $|x|$ the sum $x_1+\dots+x_n$ of its coordinates. 

\begin{lemma}\label{lem:3.1} If $\La$ is of infinite type, then $|\hat{p}_i|>|p_i|$, for any vertex $i\in Q_0$, 
where $\hat{p}_i$ denotes the vector $P_i^-=p_i^+$ corresponding to vertex $i$. \end{lemma} 

\begin{lemma}\label{lem:3.2} If $\La$ is of infinite type then there is no arrow $\alpha: i\to j$ with $i^+ = \{ \alpha\} = j^-$. 
\end{lemma} 

The following result gives a useful tool for constructing triangles in $Q$ induced from relations defining 
$\La=KQ/I$ \cite[see Proposition 4.1]{EHS1}. 

\begin{lemma}\label{lem:3.3} Assume $\alpha: i\to j$ and $\beta: j\to k$ are arrows such that $\alpha\beta \prec I$. Then there
is an arrow in $Q$ from $k$ to $i$, so that  $\alpha$ and $\beta$ are part of a triangle in $Q$. \end{lemma} 

The following result (sometimes called the {\it Triangle Lemma} \cite[Lemma 4.3]{EHS1}) shows how relations 
propagate through triangles in $Q$. 

\begin{lemma}\label{lem:3.4}  Assume $Q$ contains  a triangle
        
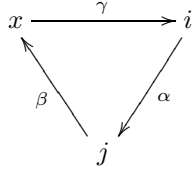
\[
 \xymatrix@R=3.pc@C=1.8pc{
    x
    \ar[rr]^{\gamma}
    && i
    \ar@<.35ex>[ld]^{\alpha}
    \\
    & j
    \ar@<.35ex>[lu]^{\beta}
  }
\]
with $\alpha\beta\prec I$. If $\gamma$ is the unique arrow $x\to i$, then $\gamma\alpha\prec I$ and 
$\beta\gamma\prec I$. If we have double arrows $\gamma,\bar{\gamma} x\to i$, then there is one 
$\delta\in\{\gamma,\bar{\gamma}\}$ such that $\delta\alpha\prec I$ and $\beta\delta\prec I$. \end{lemma} 

We have also the following result on the neighbours of $1$-vertices \cite[see Lemma 4.4]{EHS1}. 

\begin{lemma}\label{lem:3.5} Assume $i$ is a 1-vertex which is part of a triangle 

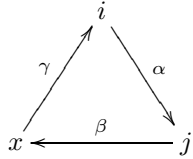
\[
  \xymatrix@R=3.pc@C=1.8pc{
    & i
    \ar[rd]^{\alpha}
    \\
  x 
    \ar[ru]^{\gamma}
    && j
   \ar[ll]_{\beta}
  }
\]

Then both $|x^-|\geqslant 2$ and $|j^+|\geqslant 2$. In particular, if $\La$ is of infinite type, 
then $x,j$ are at least $2$-vertices. \end{lemma} 

Finally, the last two results (see \cite[Propositions 4.5 and 4.6]{EHS1} describe the properties of paths of length $3$ 
involved in minimal relations. 

\begin{lemma}\label{lem:3.6} Consider the path 
$$\xymatrix{i \ar[r]^{\alpha} & k \ar[r]^{\beta} & t \ar[r]^{\gamma} & j}$$ 
such that $\alpha\beta \not\prec I$ and $\alpha$ is the unique arrow $i\to k$. If $\alpha\beta\gamma \prec I$, then 
there is an arrow $j\to i$. \end{lemma}

\begin{lemma}\label{lem:3.7} Assume $Q$ contains a square 
$$
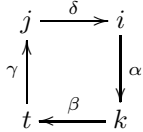
\xymatrix{j\ar[r]^{\delta} & i\ar[d]^{\alpha} \\ t\ar[u]^{\gamma} & k\ar[l]_{\beta}}$$ 
with $\alpha\beta\gamma\prec I$ but $\beta\gamma\nprec I$. If $\delta$ is the unique arrow $j\to i$, then 
$\beta\gamma\delta\prec I$. \end{lemma}  

\bigskip 

\section{Periodicity shadows}\label{sec:3} 

In this section, we introduce the main notion of this paper, the {\it periodicity shadows}, and discuss how they 
can be used in classification of the Gabriel quivers of GQT-algebras. \medskip  

Let $\Lambda=KQ/I$ be a fixed GQT algebra with Gabriel quiver $Q_\Lambda=Q$ and consider the associated (signed) 
adjacency matrix $A=\Adj_Q$, that is, its shadow $A=\bS_\La$ of $\La$. What are the properties of $A$? \medskip 

First, this is obviously a $n\times n$ skew-symmetric matrix with coefficients in $\bZ$. Using Theorem \ref{thm:motivation}, 
we conclude that there exists a symmetric matrix $C$ with natural coefficients, such that $AC=0$. Moreover, $C$ is the 
Cartan matrix of an algebra, so all its columns are nonzero. In particular, it follows that $A$ is singular. Note also 
that every skew-symmetric matrix with odd $n$ is singular, so this condition can reduce the number of possible matrices 
only in even size. \smallskip 

The second property is a bit more subtle and exploits existence of natural solutions of $AC=0$. Indeed, suppose that 
$AC=0$, for a natural symmetric matrix $C$, and let $a=[a_1 \ \dots \ a_n]$ be a non-zero row (equivalently, column) 
of $A$ with all entries $a_j\geqslant 0$ or all $a_j\leqslant 0$. Denote by $j_1,\dots,j_r$ all the indices in $\{1,\dots,n\}$, 
for which $a_{j_k}\neq 0$, and consider arbitrary column $c^T$, $c=[c_1 \ \dots \ c_n]$, of $C$. Then $ac^T=0$, so 
$a_{j_1}c_{j_1}+\dots a_{j_r}c_{j_r}=0$, and hence, we have $|a_{j_1}|c_{j_1}+\dots |a_{j_r}|c_{j_r}=0$, since all 
$a_j$'s have the same sign. But $c_i\in\bN$, thus $c_{j_1}=\dots=c_{j_r}=0$, and hence, each column of $C$ has zeros 
in all rows $j_1,\dots,j_r$. It means that rows $e_{j_k}C$ of $C$ are zero, for all $k\geqslant 1$. Consequently, $C$ 
has at least one zero row, so a zero column, by symmetricity. Therefore, we obtain that $A$ with the above property 
cannot be a shadow of a GQT algebra. \smallskip 

As a result, the matrix $A=\Adj_Q$ does not admit a non-zero row containing elements of the same sign. Rephrasing 
this in the language of quivers, this is equivalent to say that the reduced quiver $\bfQ_A$ of $A$ has no source or 
sink. Let us only mention that the Gabriel quivers of selfinjective algebras share this property. \smallskip 

We have not used tameness of $\La$ so far. Hence all the above properties hold for the adjacency matrix $A=\Adj_Q$ of 
the Gabriel quiver of each symmetric algebra $\Lambda=kQ/I$ with all simples periodic of period $4$, without restricting 
to tame or infinite representation type. \medskip 

Based on the above discussed properties, we suggest the following definition. 

\begin{defi}\label{per-sh} A matrix $A\in\bM_n(\bZ)$ is called a {\bf periodicity shadow}, if the following holds: 
\begin{enumerate}
\item[PS1)] $A$ is a singular skew-symmetric matrix. 
\item[PS2)] $A$ does not admit a non-zero row containing integers of the same sign. 
\item[PS3)] There exists a symmetric matrix $C\in\bM_n(\bN)$ with non-zero columns such that $AC=0$. 
\end{enumerate}\end{defi} 

Let us clarify that the condition PS3) implies PS2), as explained above, but from algorithmic reasons (see Introduction), 
it is better to exclude matrices not satisfying PS2) first, instead of checking PS3) for all of them. Of course, 
there may be many such periodic shadows (even with $n$ fixed), if we do not bound the range of entries. \smallskip 

Periodicity shadows in the above sense cover all adjacency matrices of symmetric algebras having all simples 
periodic of period $4$. Both wild and tame (including all GQT algebras), and hence there is barely no hope to 
classify them in general. But restricting ourselves only to shadows coming from tame algebras, there comes a severe 
simplification. Basically, if $\Lambda$ is tame, then the associated adjacency matrix $A=\Adj_Q$ satisfy 
the following properties: {\it 
\begin{enumerate}
\item[T1)] all entries $a_{ij}$ of $A$ are in $\{-2,-1,0,1,2\}$. 
\item[T2)] each row of $A$ does not contain simultainously: both $2$ and entry $\geqslant 1$ or both $-2$ and entry 
$\leqslant -1$. 
\item[T3)] each row of $A$ cannot have more than four $1$'s or more that four $-1$'s. 
\end{enumerate} } 

\begin{defi} 
We say that $A\in\bM_n(\bZ)$ is a {\bf tame periodicity shadow}, if $A$ is a periodicity shadow satisfying 
conditions T1)-T3). \end{defi}\medskip 

Instead of using full name 'tame periodicity shadow', we will frequently say just {\it shadow}. If $A=\bS_\La$ 
is a shadow of a tame symmetric algebra with $4$-periodic simples, we will always call $A$ a {\it shadow of $\La$}, 
to distinguish from shadow alone, which is a matrix satisfying PS1)-PS3) and T1)-T3). \smallskip  

Observe that T1) guarantees that we have no wild Kronecker subquiver $K_3$ in $\bfQ_A\subset Q$, while T2) and 
T3) exclude wild subquivers of types $K_2^\pm$ and $S_5^\pm$ (see Section \ref{sec:2}). One can think of enlarging 
the set of restrictions to other subquivers - shrinking the resulting lists of shadows - but this seem to be a rather 
technical issue, postponed for future. Now, we are using the simplest conditions, excluding, say 'the most wild' 
subquivers, and look what is left. It turns out that even with such a simple set of restrictions, this gives a pretty 
dense sieve. \smallskip 
 
Namely let us look at the smallest $n\in\{3,\dots,6\}$. Then, from the total number of $5^{\frac{1}{2}n(n-1)}$ 
skew-symmetric integer matrices with coefficients in $[-2,2]$ we have only 
\begin{itemize} 
\item $5$ tame periodicity shadows, for $n=3$, 
\item $12$ tame periodicity shadows, for $n=4$, 
\item $65$ tame periodicity shadows, for $n=5$, and 
\item $516$ tame periodicity shadows, if $n=6$.  
\end{itemize} 

We mean tame periodicity shadows with respect to permutation of rows and columns, i.e. we have $5,12,65,516$, so 
called {\it basic} periodicity shadows $\bS_1,\bS_2\dots$, and the rest is obtained by taking $P^T \bS_k P$, 
$k\geqslant 1$, for all permutation matrices $P$. Denote by ${\bf \bS }(n)$ the set of all $n\times n$ basic tame 
periodicity shadows. The full algorithm performing computation of ${\bf \bS}(n)$ is discussed in a separate note  
\cite{BSapx}, where we also provide lists of so called {\it essential} shadows, for $3 \leqslant n \leqslant 6$. \smallskip 

Complete lists of shadows (together with lists of related {\it shades}) can be found in an appendix \cite{BSlist}, 
from where we took the above numbers. The notion of a {\it shade} is an auxiliary notion, namely, a shade is an 
integer matrix satisfying PS1)-PS2) and T1)-T3). We only mention that the computations are based on a recursive 
procedure generating all shades first, and then verifying PS3), to left only (basic) tame periodicity shadows. 
Let us finish this issue with the following table, borrowed from \cite{BSapx}, which shows more detailed numbers 
for $n\leqslant 6$. \medskip 

\begin{tabular}{|p{2cm}|p{4cm}|p{4cm}|p{4cm}|} 
\hline 
& number of shades & number of shadows & essential shadows \\ 
\hline 
n=3 & 5 & 5 & 4 \\ 
n=4 & 12 & 12 & 7 \\ 
n=5 & 138 & 65 & 26 \\ 
n=6 & 1260 & 516 & 223 \\ 
\hline 
\end{tabular} 

\bigskip 

The rest part of this section will be devoted to discuss how the (tame) periodicity shadows are used to 
uncover the general structure of the Gabriel quiver of an algebra. \medskip

To begin with, fix an indecomposable tame symmetric algebra $\Lambda=kQ/I$, $Q_\Lambda=Q$, with all simples 
periodic of period $4$. Then the associated matrix $A=\Adj_Q$ is a tame periodicity shadow of size $n\times n$. 
Up to labeling of vertices, we can assume that $A$ is one of the basic periodicity shadows from the (finite) list 
${\bf \bS}(n)$. As a result, we have $Q\equiv \bfQ_A$, or equivalently, the reduced form of $Q$ is $Q^\times=\bfQ_A$, 
in the notation of Section \ref{sec:3}. \medskip 

We recall that $Q$ (modulo loops) has the form of a glueing: 
$$Q^\circ=\bbQ\sqcup \overline{G},$$ 
of the smallest quiver $\bbQ:=\bfQ_A$ with adjacency matrix $A$ and the double-quiver associated to a graph $G$ 
(possibly with multiple edges). Later we will prove that $G$ has a very restrictive form (see Theorem \ref{thm1}). 
Now, let us start with the following simple observation. \medskip 

\begin{prop}\label{gabriel} If $n=2$, then $Q$ is isomorphic to one of the following four quivers: 
$\newline$
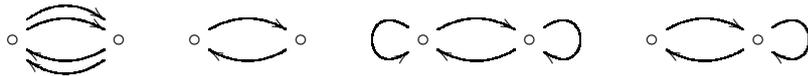
$$\xymatrix@C=0.4cm{\circ \ar@<+0.1cm>@/^10pt/[rr] \ar@/^8pt/[rr] & & \circ \ar@<+0.1cm>@/^10pt/[ll]\ar@/^8pt/[ll] } 
\xymatrix@C=0.4cm{& \circ \ar@/^8pt/[rr] && \ar@/^8pt/[ll] \circ & } 
 \xymatrix@C=0.4cm{&\ar@(lu,ld)[] \ar@/^8pt/[rr] \circ && \ar@/^8pt/[ll] \circ \ar@(ru,rd)[] & } 
\xymatrix@C=0.4cm{&\ar@/^8pt/[rr] \circ && \ar@/^8pt/[ll] \ar@(ru,rd)[] \circ & }$$
$\newline$ 
If $n\geqslant 3$, then $G$ is a simple graph. In particular, $Q$ is obtained from $Q^\circ$ by 
adding a finite number of loops (at most one per each vertex). \end{prop}

\begin{proof} We note that the unique periodicity shadow of size $n=2$ is the zero matrix $A$, with associated quiver 
$\bfQ_A$ empty. Then $Q^\circ=\overline{G}$ for a (connected) graph $G$ on two vertices. Then $G$ must have 
either two edges or one edges connecting the two vertices (cannot have more than two, due to tameness of $\Lambda$). 
In the first case, we have $Q=Q^\circ$ ($Q$ already has double arrows, so no more loop can be added), and in the second, 
$Q$ is a double cycle enlarged by at most two loops. This proves the first part. \medskip 

For the second, it is sufficient to see that $G$ must be simple, if $n\geqslant 3$. Indeed, suppose that there are 
vertices $i,j$ in $Q_0$ such that $G$ has at least two edges between $i$ and $j$. Then $Q$ admits the following 
subquiver 
$\newline$
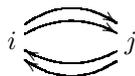
$$\xymatrix@C=0.5cm{i \ar@<+0.1cm>@/^10pt/[rr] \ar@/^8pt/[rr] & & j \ar@<+0.1cm>@/^10pt/[ll]\ar@/^8pt/[ll] }$$ 
$\newline$
Since $\Lambda$ is assumed connected, the quiver $Q$ is also connected, and hence we can find at least one vertex 
$k\neq i,j$ such that $Q$ has an arrow between $k$ and $i$ or $k$ and $j$. But then $Q$ admits a forbidden wild subquiver 
of the form $K_2^\pm $, a contradiction. \medskip 

The number of loops in $Q$ at given vertex is at most one, due to tameness, as explained in Section \ref{sec:2}. 
\end{proof} \medskip

Recall that any quiver $Q$ is uniquely determined by its arrow matrix $\Arr_Q\in\bM_n(\bN)$. In particular, $q_{ii}$ 
is the number of loops and $\Adj_Q=\Arr_Q-\Arr_Q^T$. For $n\geqslant 3$, the quiver $Q$ is given (modulo loops) by 
the matrix 
$$\Arr_{Q^\circ}=\Arr_{\bbQ}+\Arr_{\overline{G}},$$ \medskip where $\Arr_{\overline{G}}$ is a symmetric matrix with 
coefficients $0,1$ (actually, it is the usual adjacency matrix of the simple graph $G$). We will further see that 
$G$ is in fact a disjoint union of edeges (see Theorem \ref{thm1}), now let us point out that the glueing of $\bbQ$ with 
$\overline{G}$ must satisfy the following additional properties. \medskip  

\begin{enumerate}
\item $\Arr_Q$ is obtained from $\Arr_{Q^\circ}$ by adding a diagonal matrix with entries $q_{ii}\in\{0,1\}$ on the 
diagonal and $\Arr_Q$ satisfies T1)-T3). 

\item If $a_{ij}=\pm 2$, then $G$ has no edge connected to $i$ or $j$. In particular, no edge between $i$ and $j$. 

\item If $\La$ is representation-infinite, and $i,j\in Q_0$ such that $q_{ij}=1$ is the unique positive element 
in $i$-th row of $\Arr_Q$, then $q_{ij}$ is not the unique positive element in $j$-th column of $\Arr_Q$. 
\end{enumerate} 

We only mention that the third condition is an immediate consequence of Lemma \ref{lem:3.2}. Combining results 
from the next section and the above properties we may always describe all possible Garbiel quivers of GQT algebras, 
which are obtained from a finite number of shadows by attaching edges such that $\Arr_Q$ satisfy \ref{lem3}-\ref{lem6} 
and (1)-(3) above. Moreover, there are some other possibilities of restricting the list of quivers, for example, 
considering only essential shadows, but we will not say more about this here (see \cite{BSapx} for more details).  \bigskip 

\section{Proof of The Reconstruction Theorem}\label{sec:4} 

In this section we present a sequence of observations leading to the proof of the main result of this paper, i.e. 
The Reconstruction Theorem (see Introduction). As a subsequent corollary, we will also obtain a strong result 
pertaining the zero shadow case, which includes well-know case of algebras of quaternion type; see Corollary \ref{coro1} 
and comments below. \medskip 

Let $\Lambda=KQ/I$ be a fixed indecomposable tame symmetric algebra with all simple modules in $\mod\La$ 
periodic of period $4$. We will write $\mathbb{Q}$ for the quiver $\bfQ_{\bS}$ identified with its shadow 
$\bS:=\bS_\Lambda=\Adj_{Q_\Lambda}$. Then the Gabriel quiver $Q=Q_\Lambda$ of $\Lambda$ is given (up to 
loops) by its loop-free part  
$$Q^\circ=\bbQ\sqcup\overline{E},$$ 
where $E$ is a simple graph on $E_0=Q_0=\{1,\dots,n\}$ (we assume $n\geqslant 3$, since all is known for 
$n\leqslant 2$). We use letter $E$ for the graph denoted earlier by $G$ to indicate that it will be a 
disjoint union of edges. With this setup, we have the following first observation. \medskip 

\begin{lemma}\label{lem1} Every vertex in $E$ has degree at most $2$. Hence in particular, $E$ is a disjoint union 
of lines and cycles. \end{lemma}

\begin{proof} Suppose to the contrary that $E$ admits a vertex $i$ with degree $\geqslant 3$. Then $Q$ admits a 
subquiver of the form 
$$\xymatrix{& y\ar@<-0.1cm>[d]_{\beta_2} & \\ 
x \ar@<-0.1cm>[r]_{\beta_1} & i \ar@<-0.1cm>[l]_{\alpha_1} \ar@<-0.1cm>[r]_{\alpha_3}  
\ar@<-0.1cm>[u]_{\alpha_2} & z \ar@<-0.1cm>[l]_{\beta_3}}$$ 
Moreover, there is no loop $\rho:i\to i$ (in $Q$), because otherwise we would get a wild subcategory in covering 
of the form: 
$$
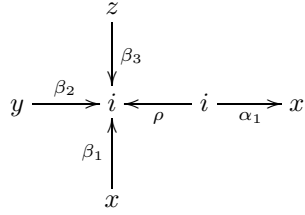
\xymatrix{&z\ar[d]^{\beta_3}&& \\ 
y\ar[r]^{\beta_2} & i & i \ar[l]^{\rho} \ar[r]_{\alpha_1} & x \\ 
&x \ar[u]^{\beta_1} && }$$ 
Thus, using Lemma \ref{lem:3.3}, we deduce that $\alpha_i\beta_i\nprec I$, for any $1\leq i \leq 3$. It follows 
that $\Lambda$ admits the following subcategory (in covering):  
$$
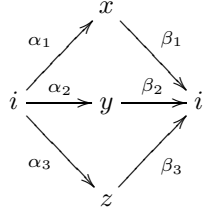
\xymatrix{&& x \ar[rd]^{\beta_1} && \\ 
& i \ar[r]^{\alpha_2} \ar[ru]^{\alpha_1} \ar[rd]_{\alpha_3} & y \ar[r]^{\beta_2} & i & \\ 
&& z\ar[ru]_{\beta_3} && }$$ 
which is isomorphic to a wild hereditary algebra, and we conclude that $\Lambda$ is wild, a contradiction. \end{proof} 

Next observation shows that connected parts of $E$ cannot be too big. \smallskip 

\begin{lemma}\label{lem2} Every connected component of $E$ has at most $3$ vertices. \end{lemma}

\begin{proof} Let $G=(G_0,G_1)$ be an arbitrary connected component of $E$, and suppose it has $m\geqslant 4$ vertices. 
We claim first that $Q$ has no loops at vertices of $G$, which are not leaves. Assume it is not the case, and let 
$i\in G_0$ be a vertex, which is not a leaf of $G$ (equivalently, of $E$), such that there is a loop $\rho:i\to i$ 
in $Q_1$. Then $Q$ admits the following subquiver: 
$$
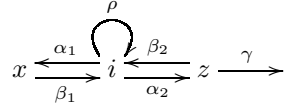
\xymatrix@R=0.35cm{ & & & \\ 
x \ar@<-0.1cm>[r]_{\beta_1} & i \ar@(ul,ur)[]^{\rho} \ar@<-0.1cm>[l]_{\alpha_1} \ar@<-0.1cm>[r]_{\alpha_2} & 
z \ar@<-0.1cm>[l]_{\beta_2} \ar[r]^{\gamma} & } $$
where all arrows besides $\rho$ are coming from edges in $G$ (which is either a line graph or a cycle with $m\geqslant 4$ 
vertices, by Lemma \ref{lem1}). But this leads to the following wild (hereditary) subcategory in covering:  
$$
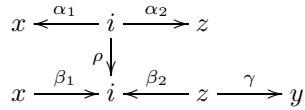
\xymatrix@R=0.45cm{ x & i\ar[l]_{\alpha_1} \ar[d]_{\rho} \ar[r]^{\alpha_2} & z & \\ 
x \ar[r]^{\beta_1} &  i  & \ar[l]_{\beta_2} z \ar[r]^{\gamma} & y }$$ 
Hence indeed, there are no loops $\rho:i\to i$ in $Q$ such that $i\in G_0$ is not a leaf. As a result, the quiver $Q$ 
has a subquiver $\overline{G}$ of one of the following two forms 
$$
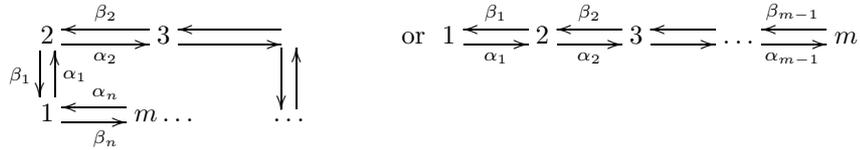
\xymatrix@R=0.6cm{ 2 \ar@<-0.1cm>[d]_{\beta_1} \ar@<-0.1cm>[r]_{\alpha_2} & 3 \ar@<-0.1cm>[l]_{\beta_2} \ar@<-0.1cm>[r] 
& \ar@<-0.1cm>[l] \ar@<-0.1cm>[d]& \\ 
1 \ar@<-0.1cm>[u]_{\alpha_1} \ar@<-0.1cm>[r]_{\beta_n} & \ar@<-0.1cm>[l]_{\alpha_n} m \dots    &  \dots \ar@<-0.1cm>[u] } \mbox{ or } 
\xymatrix@R=0.35cm{1 \ar@<-0.1cm>[r]_{\alpha_1} & 2 \ar@<-0.1cm>[l]_{\beta_1} \ar@<-0.1cm>[r]_{\alpha_2} & 
3 \ar@<-0.1cm>[l]_{\beta_2} \ar@<-0.1cm>[r] & \ar@<-0.1cm>[l] \dots \ar@<-0.1cm>[r]_{\alpha_{m-1}} & 
\ar@<-0.1cm>[l]_{\beta_{m-1}} m }$$ 
where there are possibly loops $\rho\in Q_1$ at vertices $1$ or $m$ only in the second case. \smallskip 

In the first case, applying Lemma \ref{lem:3.3}, we conclude that $\alpha_i\beta_i\nprec I$ and $\beta_i\alpha_i\nprec I$, 
for any $1\leq i \leq m$. For the second case, we have $\alpha_i\beta_i\nprec I$, if $2\leq i\leq m-1$, and 
$\alpha_1\beta_1\nprec I$, if there is no loop at $1$. Dually, we obtain $\beta_i\alpha_i\nprec I$, for any 
$1\leq i \leq m-2$, and $\beta_{m-1}\alpha_{m-1}\nprec I$, if there is no loop at $m$. \smallskip 

Next, observe that $\alpha_i\alpha_{i+1}\nprec I$ and $\beta_{i+1}\beta_i\nprec I$ for every $i$, whenever defined. 
Indeed, if this is not the case, then due to Lemma \ref{lem:3.3}, we have an arrow $\gamma: i+2\to i$ or $i\to i+2$ in 
$Q\setminus \overline{E}$, for some $1\leq i\leq m$, where $i\leq m-2$ in case $G$ is a line, and vertices are taken 
modulo $m$, in case of a cycle. We may assume $\gamma$ is an arrow $i+2\to i$ with $\alpha_i\alpha_{i+1}\prec I$, since 
the arguments in the second case are dual. Suppose first that $i+2$ is not a leaf. In this case, we have the following 
wild subcategory in covering 
$$\xymatrix@R=0.6cm{i+1 & i+2 \ar[l]_{\beta_{i+1}} \ar[r]^{\alpha_{i+2}} \ar[d]^{\gamma} & i+3 & \\ 
i-1 \ar[r]^{\alpha_{i-1}} & i & i+1 \ar[l]_{\beta_i} \ar[r]^{\alpha_{i+1}} &  i+2 }$$ 
if also $i$ is not a leaf. If it is, then $G$ is a line and $i=1$. In this case, if there is a loop $\rho:1\to 1$ 
in $Q$, then we get analogous wild subcategory replacing above $\alpha_{i-1}$ by the loop $\rho:1\to 1$ resolved in 
covering. If there is no loop at $1$, we have also $\alpha_1\beta_1\nprec I$, so we end up with another 
wild (hereditary) subcategory 
$$\xymatrix@R=0.5cm{&&  \circ\ar@{-}[d]^{\rho} &&& \\ &&  4 &&& \\ 
1\ar[r]^{\alpha_1} & 2 & \ar[l]_{\beta_2} 3 \ar[u]_{\alpha_3} \ar[r]^{\gamma} & 1 & \ar[l]_{\beta_1} 2 & \ar[l]_{\alpha_1} 1 } $$
with $\rho=\beta_4$, if $m\geqslant 5$, and $\rho=\beta_3:4\to 3$, otherwise (note that $3$ is never a leaf, so 
$\alpha_3\beta_3\nprec I$). It remains to see that if $i+2$ is a leaf of $E$, i.e. $i=m-2$, then $i$ is not and one 
can construct similar wild subcategories. Indeed, we have the following wild hereditary subcategory of type 
$\wt{\wt{\bE}}_6$
$$\xymatrix@R=0.6cm{i \ar[r]^{\alpha_{i}}& i+1 & i+2 \ar[l]_{\beta_{i+1}}  \ar[d]^{\gamma} & & \\ 
\circ \ar@{-}[r]^{\rho} & i-1 \ar[r]^{\alpha_{i-1}} & i & i+1 \ar[l]_{\beta_i} \ar[r]^{\alpha_{i+1}} &  i+2 }$$
where $\rho=\beta_{i-2}$, for $m\geqslant 5$, and if $m=4$, $\rho$ is either a resolved loop at $1$, or $\rho=\beta_1$, 
if there is no loop. \medskip 

Summing up, it has been proven that $\alpha_i\alpha_{i+1}\nprec I$ and $\beta_{i+1}\beta_i\nprec I$, 
for all $i$, if the composition is defined. \smallskip 

Now, observe that for $G$ being a cycle, $\Lambda$ admits the following wild subcategory in covering: 
$$\xymatrix@R0.5cm{& 1\ar[d]^{\alpha_1}  && \\ 
1 & \ar[l]_{\beta_1} 2 \ar[d]^{\alpha_2} & \ar[l]_{\beta_2} 3 \ar[r]^{\alpha_3} & 4 \\ &3&& }$$
hence a contradiction. The same arguments work when $G$ is a line graph with no loop at $1$ (then $\alpha_1\beta_1\nprec I$). 
Using analogous subcategory formed by arrows $\alpha_{m-3},\alpha_{m-2},\alpha_{m-1}$ and $\beta_{m-2},\beta_{m-1}$, we 
can similarily exclude the case with no loop at vertex $m$. \medskip 

Finally, suppose $G$ is a line graph and there are loops $\rho:1\to 1$ and $\sigma:m\to m$ in $Q$. In this case, 
we get the a wild hereditary (type $\wt{\wt{\bE}}_6$) subcategory of the form 
$$\xymatrix@R=0.5cm{&&  4 &&& \\ &&  3\ar[u]^{\alpha_3} \ar[d]^{\beta_2} &&& \\ 
1 & \ar[l]_{\rho} 1 \ar[r]^{\alpha_1} &  2 \ar[r]^{\alpha_2} & 3 & \ar[l]_{\beta_3} 4 \ar[r]^{\xi} & \circ } $$ 
where $\xi=\alpha_4$ if $m\geqslant 5$, and it is the loop $\xi=\sigma$ (resolved in covering), otherwise. \end{proof} \medskip 

In the zero shadow case, we have $Q^\circ=\overline{E}$, so $E$ has exactly one connected component (itself), 
therefore, we obtain the following immediate corollary from the previous two results and their proofs. \smallskip 

\begin{cor}\label{coro1} If $\bS=0$, then $n\leqslant 3$. In case $n=3$, the loop free part $Q^\circ$ of $Q$ is 
one of the following two quivers 
$$
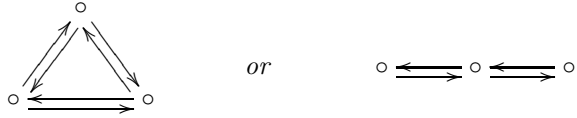
\xymatrix@C=0.5cm{ &\circ \ar[rd]\ar@<+0.13cm>[ld] & \\ \circ \ar[ru]\ar@<-0.13cm>[rr] && \circ \ar[ll]\ar@<+0.13cm>[lu]} \ 
\qquad\xymatrix@R=0.5cm{ \\ \mbox{ or } \\} 
\xymatrix@R=0.5cm{&&&& \\ 
& \circ \ar@<-0.13cm>[r] & \circ \ar[l] \ar@<-0.13cm>[r] & \circ \ar[l] & \\ 
&&&& \\}$$ 
Moreover, there are no loops in the first case, and at most one loop at each of $1,3$, in the second.  
\end{cor} \medskip 

If $C_\La$ is non-singular, then $\Adj_{Q_\La}=0$, by $(\Delta)$. Hence, as a special case of the above we get the 
following result. 

\begin{cor}\label{coro2} If $\Lambda$ has a non-singular Cartan matrix (e.g. if $\Lambda$ is an algebra of 
quaternion type), then $Q=Q_\Lambda$ has at most $3$ vertices. Moreover, $Q$ is one of the quivers described 
in \ref{gabriel} and \ref{coro1}. \end{cor} \medskip 

We mention that the above result provides an extension of known results from \cite{QT1,QT2} (see also \cite{nosim}). 
Whilst the restriction on the number of vertices is obtained by similar methods as in \cite{nosim} (based on covering techniques), 
the list of possible quivers is an immediate consequence of the tools we introduced. They provide an elementary and 
elegant alternative for relatively sophisticated machinery used in \cite{QT1,QT2}, where for instance, it was necessary 
to use Ollson's formula \cite{Ollson} for blocks with generalized quaternion defect groups to calculate the decomposition 
numbers. Moreover, our result holds in a wider class of algebras with non-singular Cartan matrices. As a consequence, we 
deduce that there are no GQT algebras with non-singular Cartan matric and more than $4$ simple modules. \medskip 

\begin{rem} \normalfont We point out that the zero shadow case gives very restrictive shape of the Gabriel quiver: 
$\Adj_Q=0$ $\Leftrightarrow$ $Q=Q_\Lambda$ is the double quiver of some graph (mostly simple) and at most $3$ 
vertices! \smallskip 

This resembles also the case of preprojective algebras, which are given by 'double quivers' (see \cite{CBH}), 
but these 'double quivers' are constructed in a slightly different manner, that is, from a quiver instead of a 
graph. We mention results of Schofield \cite{Scho} and Erdmann-Snashall \cite{ErSn}, showing that every preprojective 
algebra $\mathcal{P}_\Delta$ of Dynkin type $\Delta$ is a periodic algebra of period $1,2,3$ or $6$, and it is of tame 
representation type if and only if $\Delta=\mathbb{A}_n$, for $n\leq 5$, or $\Delta=\mathbb{D}_4$. In other words, 
for preprojective algebras, tameness implies that it cannot be too big. These algebras are given also by double 
quivers (in our sense), but the difference is that the period of simples is not $4$. \end{rem} \medskip

Now we can prove the first part of The Reconstruction Theorem. \medskip 

\begin{thm}\label{thm1} If $\bbQ$ is non-empty, then $E$ is a disjoint union of edges 
(and isolated vertices). \end{thm} 

\begin{proof} Suppose $\bbQ\neq\emptyset$ or equivalently, $\bS\neq 0$, and let $G$ be a connected component of $E$. 
By Lemma \ref{lem2}, $G$ has at most $3$ vertices, so either it is an isolated vertex, if $g=|G_0|=1$, or an edge, if 
$g=2$, and for $g=3$, it is one of the following two graphs 
$$
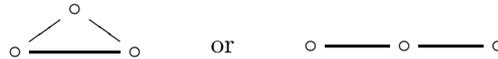
\xymatrix@R=0.2cm@C=0.4cm{ &\circ \ar@{-}[ld] & \\ \circ \ar@{-}[rr] && \circ\ar@{-}[lu]} \qquad 
\xymatrix@R=0.2cm{\\ \mbox{or} } \qquad 
\xymatrix@R=0.2cm{&& \\ \circ  \ar@{-}[r]& \circ  \ar@{-}[r] & \circ }$$ 
denoted by $F_1$ and $F_2$, respectively. We will prove that the third case is impossible by excluding both of the 
above graphs $F_1,F_2$. \medskip 

Suppose first that $G$ is a triangle $G\simeq F_1$ with the set of vertices $G_0=\{1,2,3\}$. Then $Q$ admits the 
following subquiver $\overline{G}$: 
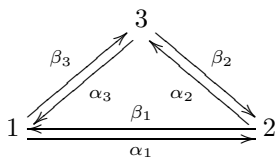
$$\xymatrix@C=1.3cm@R=1cm{ & 3 \ar[rd]^{\beta_2}\ar@<+0.13cm>[ld]^{\alpha_3} & \\ 
1 \ar[ru]^{\beta_3} \ar@<-0.13cm>[rr]_{\alpha_1} && 
2 \ar[ll]_{\beta_1}\ar@<+0.13cm>[lu]^{\alpha_2}}$$ 
Exactly as in the proof of Lemma \ref{lem2}, one can show that there is no loop in $Q$ at vertices $1,2,3$ (of $G$). 
Consequently, using Lemma \ref{lem:3.3}, we obtain that $\alpha_i\beta_i,\beta_i\alpha_i\nprec I$, for any $1\leq i \leq 3$. 
\smallskip 

Since $\bbQ$ is non-empty, there is an arrow $\gamma\in Q_1$, which is not in $\overline{G}$. 
If $Q_0=G_0$, then we would obtain a wild subquiver in $Q$ of type $K_2^\pm$, so we can assume that $Q_0\setminus G_0$ 
is non-empty. Now, because $Q$ is connected, $\overline{G}$ must be connected to the rest part of $E$ by an 
arrow from $\bbQ$, and hence we conclude that one of the vertices $i\in G_0$ is joined with some other vertex 
$x\in Q_0\setminus G_0$ by an arrow $\gamma:i\to x$ or $\gamma:x\to i$. Because the arguments are dual, we will 
restrict ourselves only to the first case, i.e. arrow $\gamma:i\to x$. In this case, we obtain the 
following wild subcategory in covering (we identify vertices modulo $3$):  
$$\xymatrix{&& x &&& &&& x & \\ 
i & \ar[l]_{\beta_i} i+1 & \ar[l]_{\alpha_i} i \ar[u]_{\gamma} \ar[r]^{\beta_{i-1}} & i-1 & 
i+1 \ar[l]_{\alpha_{i+1}} \ar[r]^{\beta_i} & i & \ar[l]_{\alpha_{i-1}} i-1 \ar[r]^{\beta_{i+1}} & i+1 & 
\ar[l]_{\alpha_i} i \ar[r]^{\beta_{i-1}} \ar[u]_{\gamma} & i-1 } $$ 
Hence indeed, $G$ cannot be a triangle, and we are done in this case. \bigskip 

Finally, let $G$ be a line $G\simeq F_2$, that is, we have a subquiver in $Q$ of the form: 
$$\overline{G}= \xymatrix{1 \ar@<-0.1cm>[r]_{\alpha_1} & 2 \ar@<-0.1cm>[l]_{\beta_1} \ar@<-0.1cm>[r]_{\alpha_2}  
& 3 \ar@<-0.1cm>[l]_{\beta_2}}$$ 
and $Q$ does not admit a loop at $2$ (see the proof of Lemma \ref{lem2}). Using Lemma \ref{lem:3.3}, we conclude 
that $\beta_1\alpha_1,\alpha_2\beta_2\nprec I$, and hence, in covering there is a tame (hereditary) subcategory 
of the following shape. 
$$\xymatrix@R=0.4cm{\\ H= \\ } \qquad  
\xymatrix@R0.4cm{ & 1 \ar[rd]^{\alpha_1}& \\ 
2 \ar[ru]^{\beta_1} \ar[rd]_{\alpha_2} && 2 \\ 
& 3 \ar[ru]_{\beta_2} & } $$ \medskip

Now, observe that we cannot have $Q_0=G_0$. Indeed, for $Q_0=G_0$ we have an arrow $\sigma\in\bbQ$, so either 
$Q$ contains a wild subquiver of type $K_2^\pm$ or $\sigma$ is an arrow between vertices $1$ and $3$. In case 
$3\to 1$ the arguments are dual, so assume we have an arrow $\sigma:1\to 3$ (but no arrow $3\to 1$). In this case, 
using Lemma \ref{lem:3.3}, we infer that $\alpha_1\alpha_2\nprec I$, so $\Lambda$ admits the following wild 
hereditary subcategory of type $\wt{\wt{\bE}}_7$ in covering.  
$$\xymatrix{ &&&& 1 &&& \\ 
3 \ar[r]^{\beta_2} & 2 & \ar[l]_{\alpha_1} 1 \ar[r]^{\sigma} & 3 & \ar[l]_{\alpha_2} 2 \ar[u]^{\beta_1} & 
\ar[l]_{\alpha_1} 1 \ar[r]^{\sigma} & 3 & \ar[l]_{\alpha_2} 2 }$$ 
if $\alpha_1\beta_1\nprec I$, and for $\alpha_1\beta_1\prec I$, we have a loop at $\rho$ at $1$, so we can get 
a wild hereditary subcategory of the following form 
$$\xymatrix@R=0.4cm{&& 2\ar[d]_{\beta_1} &&& \\ && 1 &&& \\ 
3 &\ar[l]_{\alpha_2} 2 & \ar[l]_{\alpha_1} 1 \ar[u]^{\rho} \ar[r]^{\sigma} & 3 & \ar[l]_{\alpha_2} 2 & \ar[l]_{\alpha_1} 1}$$ 

Hence $Q_0\setminus G_0$ non-empty. Because $Q$ is connected, there is an arrow $\sigma:i\to j$, 
with one of $i,j$ in $\{1,2,3\}$ and the second does not belong to $G_0$. If $i=2$ or $j=2$, then $H$ may be 
extended to the one of the following two wild subcategories: 
$$\xymatrix@R0.4cm{& & 1 \ar[rd]^{\alpha_1}& \\ 
j & \ar[l]_{\sigma} 2 \ar[ru]^{\beta_1} \ar[rd]_{\alpha_2} && 2 \\ 
& & 3 \ar[ru]_{\beta_2} & } \qquad \xymatrix@R=0.4cm{\\ \mbox{or} \\ } \qquad 
\xymatrix@R0.4cm{ & 1 \ar[rd]^{\alpha_1}& & \\ 
 2 \ar[ru]^{\beta_1} \ar[rd]_{\alpha_2} && 2 & \ar[l]_{\sigma} i \\ 
 & 3 \ar[ru]_{\beta_2} & & }$$ 
Let now $\sigma$ be an arrow $i\to j$ with $i=1$ or $3$ (and $j\notin G_0$). If $i=1$, then $\beta_1\sigma\nprec I$, 
since otherwise, due to Lemma \ref{lem:3.3}, we would get an arrow $j\to 2$, and we are in the previous case. 
Similarily, if $i=3$, then $ \alpha_2\sigma\nprec I$. As a result, if there is an arrow $\sigma:i\to j$ with 
$i=1$ or $3$, then $\Lambda$ admits one of the following wild subcategories (in covering): 
$$
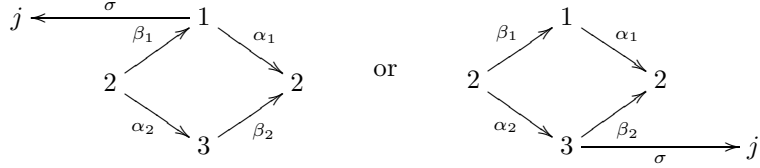
\xymatrix@R0.4cm{j &  & \ar[ll]_{\sigma} 1 \ar[rd]^{\alpha_1}& \\ 
 &  2 \ar[ru]^{\beta_1} \ar[rd]_{\alpha_2} && 2 \\ 
& & 3 \ar[ru]_{\beta_2} & } \qquad \xymatrix@R=0.4cm{\\ \mbox{or} \\ } \qquad 
\xymatrix@R0.4cm{ & 1 \ar[rd]^{\alpha_1}& & \\ 
 2 \ar[ru]^{\beta_1} \ar[rd]_{\alpha_2} && 2 &   \\ 
 & 3 \ar[ru]_{\beta_2} \ar[rr]_{\sigma} & & j }$$ 
Using dual arguments, one can prove that an arrow $\sigma:i\to j$ with $j=1$ or $3$ leads to analogous wild 
subcategories. This completes the proof. \end{proof} \medskip 

Now, we will investigate which vertices of $Q$ are joined by an edge in $E$. Let $e:\xymatrix{i\ar@{-}[r] & j}$ 
be a fixed edge of $E$ and $\alpha:i\to j$, $\beta:j\to i$ the arrows in $Q_1$ induced from $e$. Moreover, we 
denote by $\alpha_1,\dots,\alpha_q$ all the arrows in $\bbQ$ starting at $i$, where $\alpha_k:i\to a_k$, and 
by $\beta_1,\dots,\beta_p$ all the arrows in $\bbQ$ ending at $i$ with $\beta_k:b_k\to i$. As a result, $i$ 
is a $(p,q)$-vertex of $\bbQ$. Similarily, we have $r$ arrows $\delta_k: d_k\to j$ in $\bbQ$ ending at $j$, for 
$k\in\{1,\dots,r\}$, and $s$ arrows $\gamma_k:j\to c_k$ in $\bbQ$ starting at $j$, for $k\in\{1,\dots,s\}$, so 
that $j$ is a $(r,s)$-vertex of $\bbQ$. \smallskip 

\begin{lemma}\label{lem3} $i,j$ are at most $2$-regular vertices in $\bbQ$. \end{lemma} 

\begin{proof} Indeed, if $q\geqslant 3$, then there are at least $4$ arrows in $Q$ starting from $i$ 
($\alpha_1,\dots,\alpha_q\in\bbQ$ and $\alpha$ not in $\bbQ$). Moreover, then there is no loop at $i$, 
since otherwise, we would get a star $S_5$ in covering (then $\Lambda$ is wild). Consequently, applying 
Lemma \ref{lem:3.3}, we obtain $\alpha\beta\nprec I$ (and $\beta\alpha\nprec I$, if there is no loop 
at $j$). But then $\Lambda$ admits a wild subcategory (in covering) of the form. 
$$
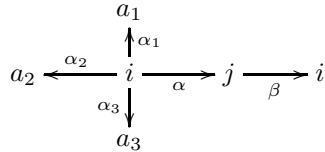
\xymatrix@R=0.4cm{ & a_1 && \\ 
a_2 & i \ar[l]_{\alpha_2} \ar[u]_{\alpha_1} \ar[d]_{\alpha_3} \ar[r]_{\alpha} & \ar[r]_{\beta} j & i \\ 
& a_3 && } $$ 
This proves that $q\leqslant 2$. If $p\geqslant 3$, we get a wild subcategory of the same type (but with 
reversed arrows), so indeed $p,q\leqslant 2$, i.e. $i$ is at most $2$-regular. Similarily, we prove that 
$j$ is at most $2$-regular, that is $r,s\leqslant 2$. \end{proof} \medskip 

In the rest part of this section we assume that the algebra $\La$ is of infinite representation type. 

\begin{lemma}\label{lem4} If one of $i,j$ is a $2$-vertex in $\bbQ$, then the second one is either isolated or also 
a $2$-vertex in $\bbQ$. In the second case, the quiver $Q$ has the following form 
$$
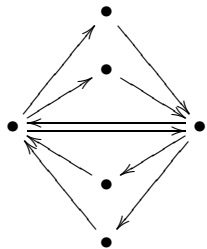
\xymatrix@R=0.4cm{& \bullet \ar[rdd] & \\ & \bullet \ar[rd] & \\ 
\bullet \ar[ru]\ar[ruu] \ar@<-0.35ex>[rr] && \ar[ld]\ar[ldd] \ar@<-0.35ex>[ll] \bullet \\ 
& \bullet \ar[lu] & \\ & \ar[luu] \bullet & }$$ \end{lemma}

\begin{proof} Without loss of generality, let $i$ be a $2$-vertex in $\bbQ$. Then $j$ is an $(r,s)$-vertex in 
$\bbQ$ with $r,s\leqslant 2$, by Lemma \ref{lem3}. We will show that either $r=s=0$ or $r=s=2$. \smallskip 

First, observe that there is no loop at $i$, because a loop $\rho\in Q_1$ at $i$ gives the following wild 
subcategory in covering: 
$$\xymatrix@R=0.4cm{ & a_1 && \\ 
a_2 & i \ar[l]_{\alpha_2} \ar[u]_{\alpha_1} \ar[d]_{\alpha} \ar[r]_{\rho} &  i & \ar[l]_{\beta_1} b_1 \\ 
& j && } $$ 
Consequently, using Lemma \ref{lem:3.3} we get that $\alpha\beta\nprec I$. Note also that $a_1\neq a_2$ and $b_1\neq b_2$, 
since otherwise we would get a wild subquiver of type $K_2^\pm$. \medskip 

Further, observe that $\alpha\beta\nprec I$ yields the following tame hereditary subcategory of type 
$\wt{\bD}_6$: 
$$(*) \qquad 
\xymatrix@R=0.4cm{& a_2 && b_2 \ar[d]^{\beta_2} & \\ 
a_1 & \ar[l]_{\alpha_1} i \ar[u]_{\alpha_2} \ar[r]^{\alpha} & j \ar[r]^{\beta} 
& i & \ar[l]_{\beta_1} b_1 }$$ \smallskip 

In particular, it follows that $a_k^-=b_l^+=1$, for any $k,l$, since otherwise, the above subcategory can 
be extended to a wild subcategory of type $\wt{\wt{\bD}}_6$. Applying the same argument, one can show that the 
following condition holds: 

{\it $(\nabla)$ 
for any arrow $\eta\in a_k^+$ ($\sigma\in b_l^-$), we have $\alpha_k\eta\prec I$ ($\sigma\beta_l\prec I$).} \medskip 

Assume now that one of $a_k,b_l$ is not $1$-regular (in $Q$). Let $a=a_1$ satisfy $|a^+|\geqslant 2$, so that 
$a^+=\{\eta_1,\dots,\eta_t\}$, for $t\geqslant 2$. 

(1) First, observe that the targets of arrows $\eta_1,\dots,\eta_t$ must be pairwise different. Indeed, otherwise 
$a^+=\{\eta_1,\eta_2\}=b^-$, where $b=t(\eta_1)=t(\eta_2)$, hence every path in $Q$ from $i$ to $b$ has the form 
$u\alpha\eta_1$ or $u\alpha\eta_2$, for some path $u$, where $\alpha=\alpha_1$. On the other side, one of the paths 
$\alpha\eta_1$ or $\alpha\eta_2$ must be involved in a minimal relation of $I$, say $\alpha\eta\prec I$, and we get 
$$\alpha\eta=u\alpha\eta+v\alpha\bar{\eta},$$ 
where $\eta:=\eta_1$, $\bar{\eta}:=\eta_2$, and we can assume $u=0$ (otherwise, adjust $\alpha:=\alpha-z_1\alpha$). 
Further, we have $v\in J$, since $v\notin J$ forces a relation $\alpha\eta+\lambda\alpha\bar{\eta}=0$, $\lambda\in K^*$, 
and after adjusting $\eta,\bar{\eta}$ we get a new set of relations with $\alpha\eta=0$, while 
$\alpha\bar{\eta}\nprec I$, a contradiction with $(\nabla)$. Now, having $\alpha\eta=v\alpha\bar{\eta}$ with $v\in J$, 
we deduce that $\alpha\bar{\eta}\nprec I$. Indeed, if $\alpha\bar{\eta}\prec I$ would yield a relation 
$\alpha\bar{\eta}=u\alpha\eta+v'\alpha\bar{\eta}$ with $v'\in J$, which gives a contradiction: 
$$\alpha\bar{\eta}= u\alpha\eta+v'\alpha\bar{\eta} = (uv+v')\alpha\bar{\eta}=(uv+v')^2\alpha\bar{\eta} = \dots = 0, \ 
\mbox{so also } \alpha\eta=0,$$ 
and hence $\alpha\in\soc(\Lambda_\La)$. But $\alpha\bar{\eta}\nprec I$ contradicts $(\nabla)$, and therefore, we 
have proved that arrows $\eta_1,\dots,\eta_t$ have pairwise distinct targets. \smallskip 

(2) It follows from $(\nabla)$ that $t=2$, since any path $\alpha\eta_s\prec I$ gives $t(\eta_s)\in s(i^-)=\{b_1,b_2,j\}$, 
by Lemma \ref{lem:3.3}, and hence, if $t\geqslant 3$, then $t=3$, by (1), so using identity $p_a^-=p_a^+$, 
we get 
$$p_i=p_a^-=p_a^+= p_{t(\eta_1)}+p_{t(\eta_2)}+p_{t(\eta_3)}=p_{b_1}+p_{b_2}+p_{b_3}=p_i^-,$$ 
which contradicts Lemma \ref{lem:3.1}. Consequently, we obtain $t=2$, and again from relations in $(\nabla)$, we 
infer that $t(\eta_1)=b_1$ and $t(\eta_2)=b_2$ (up to relabelling). Moreover, using $p_a^-=p_a^+$, we deduce that 
$p_i=p_{b_1}+p_{b_2}$. But then $a_2$ must be also a $(1,2)$-vertex, since if it was a $1$-vertex, the unique arrow 
in $a_2^+$ would have target in $\{b_1,b_2\}$, by $(\nabla)$ and \ref{lem:3.3}, and this gives a contradiction: 
$p_{b_1}+p_{b_2}=p_i=p_{a_2}^-=p_{a_2}^+=p_{b_1}$ or $p_{b_2}$. \smallskip 

(3) Repeating the above argumentation for $a=a_2$, we also conclude that it has $a^+=\{\rho_1,\rho_2\}$ with both 
$\alpha_2\rho_1,\alpha_2\rho_2\prec I$, and $t(\rho_1)=b_1$ and $t(\rho_2)=b_2$. It follows also that $b_1,b_2$ 
are $(2,1)$-vertices (in $Q$) and $Q$ admits the following block  
$$\xymatrix@R=0.3cm{ \\ \\ \Gamma= \\ \\ \\} \qquad 
\xymatrix{^{b_2} \bullet \ar@/_30pt/[rdd] && \ar[ll] \bullet^{a_2} \ar[lld] \\ 
^{b_1}\bullet \ar[rd] && \bullet^{a_1} \ar[ll] \ar[llu] \\ 
& \circ_i \ar[ru] \ar@/_30pt/[ruu] & }$$ 
In particular, we have no arrows from $a_k$ to $j$ and from $j$ to $b_l$. Applying Lemma \ref{lem:3.3}, we obtain 
that $\beta\alpha_k\nprec I$ and $\beta_l\alpha\nprec I$ for all $k,l$. \smallskip 
 
Summing up the above discussion (1)-(3), we conclude that exaclty one of the following conditions hold: 
\begin{enumerate}
\item[(a)] either one of vertices $a_k,b_l$ is non-regular, and then $a_1,a_2$ are $(1,2)$-vertices, $b_1,b_2$ are 
$(2,1)$-vertices and $Q$ contains a block of the form $\Gamma$; 
\item[(b)] or all vertices $a_k,b_l$ are $1$-regular  
\end{enumerate} 
In both cases, all paths of length two are involved in minimal relations as in $(\nabla)$. \bigskip 

Now, we will prove that $j$ is isolated in the first case, or is $2$-regular, in the second. Indeed, in case 
(a) we have $\beta\alpha_k\nprec I$ and $\beta_l\alpha\nprec I$ for all $k,l$. Suppose to the contrary that there 
is an arrow $\gamma:j\to c$. By Lemma \ref{lem:3.3} we obtain $\alpha\gamma\nprec I$, since otherwise, one gets an 
arrow $c\to i$, hence $c=b_i$, a contradiction. As a result, we end up with the following wild hereditary subcategory 
(of type $\wt{\wt{\mathbb{D}}}_5$). 
$$\xymatrix@R=0.6cm{& a_2 & c & & \\ 
a_1 & \ar[l]_{\alpha_1} i \ar[u]_{\alpha_2} \ar[r]^{\alpha} & j \ar[u]^{\gamma} \ar[r]^{\beta} 
& i & \ar[l]_{\beta_1} b_1 }$$ 
Thus we cannot have an arrow $\gamma:j\to c$, i.e. $s=0$. Because $\bbQ$ has no sinks, we conclude that also $r=0$, 
and hence $j$ is indeed an isolated vertex of $\bbQ$ in case (a). \medskip 

Finally, assume that all vertices $a_k,b_k$ are $1$-regular, and fix arrows $\eta_k,\sigma_l$ with $a_k^+=\{\eta_k\}$ 
and $b_l^-=\{\sigma_l\}$. By $(\nabla)$, we obtain that $\alpha_k\eta_k\prec I$ and $\sigma_l\beta_l\prec I$, for 
any $k,l$, hence applying Lemma \ref{lem:3.3}, we deduce that there are arrows $t(\eta_k)\to i$ and $i\to s(\sigma_l)$. 
It follows that $t(\eta_k)\in\{b_1,b_2,j\}$. But $t(\eta_k)$ cannot be a $1$-vertex, since $s(\eta_k)=a_k$ 
is already $1$-regular, and we would obtain a contradiction with Lemma \ref{lem:3.2}. Consequently, we obtain 
that both $t(\eta_1)=t(\eta_2)=j$, and hence, the two arrows $\eta_1,\eta_2$ starting at $a_1,a_2$ are both ending 
at $j$. This means $j^-=r=2$ and (up to labelling) $d_1=a_1$ and $d_2=a_2$. Applying dual arguments for $s(\sigma_l)$, 
we can similarily prove that $s(\sigma_1)=s(\sigma_2)=j$, so $s=j^+=2$ and $b_1=c_1$ and $b_2=c_2$. As a result, 
it has been proved that in case (b), vertex $j$ is $2$-regular and $Q$ has the following block 
$$ \xymatrix@R=0.3cm{\\ \\ \\ \Gamma'= \\ \\ \\} \ 
\xymatrix@R=0.4cm{& \bullet^{a_1} \ar[rdd] & \\ & \bullet^{a_2} \ar[rd] & \\ 
_{i}\bullet \ar[ru]\ar[ruu] \ar@<-0.35ex>[rr] && \ar[ld]\ar[ldd] \ar@<-0.35ex>[ll] \bullet_{j} \\ 
& \bullet_{b_1} \ar[lu] & \\ & \ar[luu] \bullet_{b_2} & }$$  
Since $i,j$ are $2$-regular in $\bbQ$ and $a_k=d_k$, $b_l=c_l$ are $1$-regular in $Q$ (see also Lemma \ref{lem3}), 
we conclude that there are no more arrows in $Q$, that is, $Q=\Gamma'$ has the required form. \end{proof}

\medskip 

\begin{lemma}\label{lem5} If $i$ is a $(1,2)$-vertex and $j$ is not isolated, then $j$ is a $(2,1)$-vertex (in $\bbQ$), 
and $Q$ admits the following block 
$$
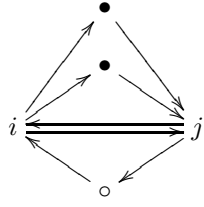
\xymatrix@R=0.4cm{ &\bullet \ar[rdd] & \\ & \bullet\ar[rd]& \\ 
i \ar[ru] \ar[ruu] \ar@<-0.35ex>[rr]&& j \ar[ld] \ar@<-0.35ex>[ll]\\ & \circ \ar[lu] & }$$ 
\end{lemma}

\begin{proof} Suppose $i$ be is a $(1,2)$-vertex and $j$ is not isolated. By Lemma \ref{lem4}, $j$ cannot be a 
$2$-vertex. As $q=2$, we cannot have a loop $\rho$ at $i$, because otherwise, we get a wild hereditary subcategory 
of type $\wt{\wt{\bD}}_4$, as in the previous proof. In particular, it follows that $\alpha\beta\nprec I$, due to 
\ref{lem:3.3}. Similarily, we cannot have loop $\rho$ at $j$, because then we would get a wild subcategory given 
as follows. 
$$
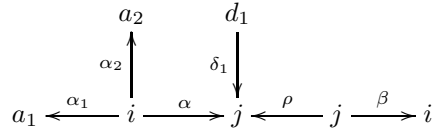
\xymatrix{& a_2 & d_1\ar[d]_{\delta_1} && \\ 
a_1 & \ar[l]_{\alpha_1} i \ar[u]^{\alpha_2} \ar[r]^{\alpha} & j & \ar[l]_{\rho} j \ar[r]^{\beta}  & i }$$ 
So also $\beta\alpha\nprec I$. \medskip

We will show below that $|j^+|=1$ and $|j^-|=2$ (in $\bbQ$). \smallskip 

Assume first that $j^+=\{\gamma_1,\gamma_2\}$ in $\bbQ$ with $\gamma_1\neq \gamma_2$. In this case, both 
$\alpha\gamma_1\prec I$ and $\alpha\gamma_2\prec I$, since one $\alpha\gamma_k\nprec I$ yields the following 
wild subcategory. 
$$
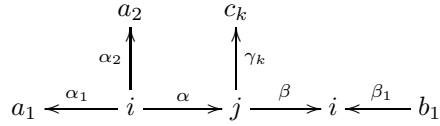
\xymatrix{& a_2 & c_k && \\ 
a_1 & \ar[l]_{\alpha_1} i \ar[u]^{\alpha_2} \ar[r]^{\alpha} & j \ar[r]^{\beta} \ar[u]_{\gamma_k} & i & b_1\ar[l]_{\beta_1} }$$ 
Now, $\alpha\gamma_1,\alpha\gamma_2\prec I$ implies  that $c_1=c_2=b_1$, by Lemma \ref{lem:3.3}, because there are no 
loops at $j$, which is impossible, since then we have doube arrows between $j$ and $c_1=c_2$, and an arrow 
$\beta:j\to i$, so we obtain a wild subquiver of type $K_2^\pm$, a contradiction. This shows that indeed $j^+=1$. 
By the proof, we may assume $\alpha\gamma_1\prec I$, so again $b_1=c_1$. \medskip 

Next, suppose $j^-=1$. Then $j$ is a $1$-vertex in $\bbQ$ and $\beta,\gamma_1$ are all arrows in $Q$ starting at $j$. 
Similarily, all arrows ending at $j$ are $\delta_1,\alpha$, thus applying $p_j^-=p_j^+$, we obtain $p_d=p_c$, where 
we abbreviate $d=d_1$ and $c=c_1$ (common term $p_i$ cancels on both sides). In the same way, using $p_i^-=p_i^+$, 
we get that $p_b=p_{a_1}+p_{a_2}$, where $b=b_1$. We have also an equality of vertices $c=b$, so we conclude that 
$$\leqno{(*)}\qquad\qquad p_d=p_c=p_b=p_{a_1}+p_{a_2}.$$ 
But one of $\beta\alpha_1\prec I$ or $\beta\alpha_2\prec I$, since otherwise we would have found the following wild 
subcategory. 
$$
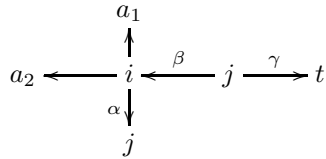
\xymatrix@R=0.4cm{ & a_1 && \\ a_2 & i \ar[l] \ar[u] \ar[d]_{\alpha} & \ar[l]_{\beta} j \ar[r]^{\gamma} & t\\ & j && } $$ 
Hence indeed, we have $\beta\alpha_1\prec I$ or $\beta\alpha_2\prec I$, so by Lemma \ref{lem:3.3}, we get an arrow 
$a_k\to j$. We have no loops at $i$, thus $a_k\neq i$, and therefore, we get $a_k=d$. But then the identity $(*)$ 
implies $p_{a_{1}}=0$ or $p_{a_2}=0$, a contradiction. \medskip 

As a result, we proved that $j$ is indeed a $(2,1)$-vertex. Now, observe that there is the following tame subcategory 
in covering of $\Lambda$: 
$$\xymatrix{& a_2 & d_1\ar[d]_{\delta_1} & \\ 
a_1 & \ar[l]_{\alpha_1} i \ar[u]^{\alpha_2} \ar[r]^{\alpha} & j & \ar[l]_{\delta_2} d_2 }$$ 
This subcategory can be extended to a wild subcategory by one arrow starting or ending at $a_{1,2}$ or $d_{1,2}$. 
Hence we conclude that $a_{1,2}^-=1$ and $d_{1,2}^+=1$. Further, note that both $\beta\alpha_1,\beta\alpha_2\prec I$, since 
for one $\beta\alpha_k\nprec I$, we can construct the following wild subcategory. 
$$\xymatrix@R=0.6cm{& & a_k & d_1\ar[d]_{\delta_1} &  \\ 
t & \ar[l]_{\gamma_1} j \ar[r]^{\beta} & i \ar[u]_{\alpha_k}  \ar[r]^{\alpha}&  j  & \ar[l]_{\delta_2} d_2 }$$ 
Therefore, we deduce from Lemma \ref{lem:3.3} that $\{a_1,a_2\}=\{d_1,d_2\}$, so $Q$ admits the required subquiver. 
Moreover, both $a_1=d_1$ and $a_2=d_2$ are $1$-vertices (in $Q$), because otherwise $a_{1,2}^-\geqslant 2$ or 
$d_{1,2}^+\geqslant 2$. \medskip 

Finally, note that there are no more arrows in $Q$ starting or ending at $i,j$ and $b=c$, so the subquiver is 
a block. The proof is now complete. \end{proof} \medskip 

Similarily, one can prove dual version of the above lemma: {\it if $i$ is a $(2,1)$-vertex, then $j$ is a 
$(1,2)$-vertex, and we obtain analogous subquiver.} 

\begin{remk} \normalfont 
In the last proof presented below, we will sometimes refer 
to the Ringel's list \cite[see 1.5. Theorem 2]{Rin}, which contains 37 wild one-relation algebras. We will 
write RiN to indicate the wild one-relation algebra, whose number on the Ringel's list is $N$. Note that the 
list in \cite{Rin} uses three sorts of numerations to denote the graphs. By a number on this list, we mean its 
number in the Roman numeration (one of the three conventions). 
\end{remk} \smallskip 

\begin{lemma}\label{lem6} If both $i,j$ are $1$-regular, they are contained in one component of $\bbQ$, which induces 
the following block in $Q$.  
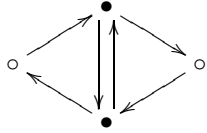
$$\xymatrix@R0.4cm{ & \bullet \ar[rd] \ar@<-0.1cm>[dd] & \\ 
\circ \ar[ru]   && \circ \ar[ld]\\ 
& \bullet \ar@<-0.1cm>[uu] \ar[lu] & }$$
\end{lemma}

\begin{proof} Assume $i,j$ are both $1$-vertices of $\bbQ$; we keep the notation introduced before Lemma \ref{lem3}, 
but we abbreviate $a:=a_1$, $b:=b_1$, $c=c_1$ and $d=d_1$. First, we claim that there are no loops in $Q$ at vertices 
$i$ and $j$. If $\rho\in Q_1$ is a loop at $i$, we can find the following wild subcategory in covering: 
$$\xymatrix@R=0.6cm{&& b \ar[d]^{\beta_1} & a & \\ 
c & \ar[l]_{\gamma_1} j \ar[r]^{\beta} & i & \ar[l]_{\rho} i \ar[r]^{\alpha} \ar[u]^{\alpha_1} & j }$$ 
Analogous wild subcategory of type $\wt{\wt{\bD}}_5$ can be found by resolving a loop at $j$. Hence indeed, we have 
no loops at $i,j$. In particular, due to Lemma \ref{lem:3.3}, we conclude that $\alpha\beta\nprec I$ and 
$\beta\alpha\nprec I$. Therefore, all the arrows starting or ending at $i$ or $j$ are $\alpha,\beta,\alpha_1,\beta_1,
\gamma_1$ and $\delta_1$. To prove our claim it is sufficient to show that $b=c$ and $a=d$. \medskip 

Note that identities $p_i^-=p_i^+$ and $p_j^-=p_j^+$ yield equalities of dimension vectors: $p_b=p_a$ and $p_d=p_c$ 
(common terms $p_i$ or $p_j$ cancel on both sides). \bigskip 

Suppose now that $b\neq c$. Then using Lemma \ref{lem:3.3}, we deduce that $\beta_1\alpha\nprec I$ and 
$\alpha\gamma_1\nprec I$. Moreover, we must have $|b^+|=1$ or $|c^-|=1$. Indeed, suppose both $|b^+|,|c^-|\geqslant 2$. 
Then there are arrows $\beta':b\to\circ$, different from $\beta_1$, and $\gamma':\circ\to c$, different from $\gamma_1$, 
and hence, we obtain the following wild hereditary subcategory of type $\wt{\wt{\bE}}_6$. 
$$\xymatrix@R=0.5cm{&& d \ar[d]_{\delta_1}  &&& \\ && j  &&& \\ 
\circ & \ar[l]_{\beta'} b \ar[r]^{\beta_1} & i \ar[u]^{\alpha} & \ar[l]_{\beta} 
j \ar[r]^{\gamma_1} & c & \ar[l]_{\gamma'} \circ }$$ 
Consequently, we proved that one of $|b^+|$ or $|c^-|$ is $1$, say $|b^+|=1$ (the proof in case $|c^-|=1$ follows dually). 
In particular, $\beta_1$ is the unique arrow in $Q$ starting at $b$. We will consider two cases below. \smallskip 

{\bf Case 1.} Assume that both $\beta_1\alpha_1\prec I$ and $\delta_1\gamma_1\prec I$. Then we have two arrows 
$\sigma:a\to b$ and $\tau:c\to d$, due to Lemma \ref{lem:3.3}, so two triangles containing vertices $a,b,i$ and 
$c,d,j$, respectively. \smallskip 

Suppose first that $\beta\alpha_1\nprec I$. In this case, $\beta_1\alpha_1$ is the unique path of length two passing 
through $i$ which is involved in a minimal relation. If $|d^+|\geqslant 2$ and $\delta:d\to \circ $ is an arrow in 
$d^+$ different from $\delta_1$, then we get the following subcategory in covering 
$$\xymatrix@R=0.6cm{ & a && b\ar[ld]_{\beta_1} && \\ 
c & \ar[l]_{\gamma_1} j \ar[r]^{\beta} & i \ar[lu]_{\alpha_1} \ar[r]^{\alpha} & j & \ar[l]_{\delta_1} d \ar[r]^{\delta} & \circ }$$
But it is isomorphic to a wild one-relation algebra RiII, and we obtain a contradiction. \smallskip 

If $|d^+|=1$, then we can construct a subcategory of the same shape, except $\delta$ is an arrow in $d^-$ with 
$\delta\delta_1\nprec I$. Indeed, let $|d^+|=1$ and consider arbitrary arrow $\delta\in d^-$. We can assume that 
$|d^-|\geqslant 2$ and $\delta\neq \tau$, because otherwise $d$ is a $1$-vertex in $Q$, and hence $p_d=p_c=p_d^-$, 
so we obtain a contradiction with Lemma \ref{lem:3.1}. Now, given $\delta:d'\to d$ with $\delta\delta_1\prec I$, 
we get an arrow $j\to d'$, by \ref{lem:3.3}, so $d'=i$ or $c$. If $d'=c$, then we get double arrows from $c$ to $d$, 
and both $\tau\delta_1\prec I$ and $\delta\delta_1\prec I$. Actually, we can change presentation as described in 
the proof of Lemma \ref{lem4} to get one of the paths $\tau\delta_1\nprec I$ or $\delta\delta_1\nprec I$, and take 
$\delta$ as the one not inducing a relation. Hence we can assume that $d'=i$, so we have $d=a$ and $\delta=\alpha_1$ 
($d\neq j$, since there are no loops at $j$). Now, we claim that there is one more arrow $\xi\in d^-$, different 
from $\tau$ and $\delta$. Otherwise, we would get $d^-=\{\tau,\delta\}$, thus the following identities hold: 
$$p_j=p_d^+=p_d^-=p_c+p_i=p_j^+,$$ 
and we obtained a contradiction with Lemma \ref{lem:3.1}. As a result, there is an arrow $\xi:x\to d$ such 
that $\xi\neq \tau$ and $\xi\neq\delta$. As above, we can easily exclude double arrows $c\to d$ (or $i\to d$), 
and hence $x\neq c$ and $x\neq i$. It follows from Lemma \ref{lem:3.3} that $\xi\delta_1\nprec I$ and we 
choose new $\delta:=\xi$. \smallskip 

This finishes the proof in case $\beta\alpha_1\nprec I$. Now, suppose that $\beta\alpha_1\prec I$. Then we 
have an arrow $a\to j$, so $a=d$ ($a\neq i$, because there are no loops at $i$). Consequently, we have two 
arrows $\tau\neq \alpha_1$ ending at $d$ and two arrows $\sigma\neq\delta_1$ starting at $a$. We recall 
that $|b^+|=1$, by the assumption, so $|b^-|\geqslant 2$, since otherwise, $b$ is a $1$-vertex with $p_b=p_a=p_b^-$, 
a contradiction with Lemma \ref{lem:3.1}). We will prove that for any arrow $\beta'\in b^-$, we have 
$\beta'\beta_1\nprec I$. Indeed, if $\beta'\beta_1\prec I$, then by Lemma \ref{lem:3.3}, there is an arrow 
$i\to b'=s(\beta')$, so $b'=a$ or $j$. But $b'=j$ forces $b=c$ or $i$, which is impossible by the assumption $b\neq c$ 
(and no loops at $i$), hence $b'=a$, and we obtain double arrows $\sigma,\beta':a\to b$. Now, since $a=d$ there is an 
arrow $\delta_1:a\to j$ different from $\sigma$ and $\beta'$, and we obtain a wild subquiver of type $K_2^+$, 
a contradiction. \smallskip 

Finally, take any $\beta':b'\to b$ different from $\sigma$. We know that $\beta'\beta_1\nprec I$, and this gives the 
following wild hereditary subcategory of type $\wt{\wt{\bE}}_7$. 
$$\xymatrix@R=0.6cm{&&&& c &&& \\ 
\circ & \ar[l] c \ar[r]^{\tau} & d=a & \ar[l]_{\alpha_1} i \ar[r]^{\alpha} & j \ar[r]^{\beta} \ar[u]^{\gamma_1}
& i & \ar[l]_{\beta_1} b & b' \ar[l]_{\beta'}}$$
if $|c^+|\geqslant 2$ ($\tau'\neq \tau$ in $c^+$). As before, $c$ cannot be a $1$-vertex, hence we have 
$|c^+|\geqslant 2$ if only $|c^-|=1$. For $|c^-|\geqslant 2$, we can construct a wild subcategory of the same 
type given as follows. 
$$\xymatrix@R=0.6cm{&&& b\ar[d]_{\beta_1} & &&&\\ 
\circ\ar[r] & c & \ar[l]_{\gamma_1} j \ar[r]^{\beta} & i \ar[r]^{\alpha} 
& j & \ar[l]_{\delta_1} d=a \ar[r]^{\sigma} & b & \ar[l] \circ}$$ \medskip 

{\bf Case 2.} Now, it remains to consider the case when one of $\beta_1\alpha_1,\delta_1\gamma_1$ is not involved 
in a minimal relation of $I$. We will proceed only in case $\beta_1\alpha_1\nprec I$, since the second case follows 
from dual arguments. \smallskip 

Assume that $\beta_1\alpha_1\nprec I$. In this case, we must have $\beta\alpha_1\prec I$, because otherwise, 
all paths of length two passing through $i$ are not involved in minimal relations, and we obtain the following 
wild subcategory of type $\wt{\wt{\bD}}_4$. 
$$\xymatrix@R=0.4cm{& j \ar[d]_{\beta} && \\ 
b \ar[r]^{\beta_1} & i \ar[d]_{\alpha_1} \ar[r]^{\alpha} & j & \ar[l]_{\delta_1} d  \\ & a && }$$ 
As a result $\beta\alpha_1$ is the unique path (of length two) passing through $i$ which is involved in 
a minimal relation. By Lemma \ref{lem:3.3}, there is an arrow $a\to j$, so $a=d$. Consequently, if 
$|a^+|\geqslant 2$ and $|a^-|\geqslant 2$, then we get the following subcategory 
$$\xymatrix@R=0.6cm{c & \ar[l]_{\gamma_1} j\ar[rd]^{\beta} &   b\ar[d]^{\beta_1} &&& \\ 
\circ \ar[r] & a & \ar[l]_{\alpha_1} i \ar[r]^{\alpha} & j & \ar[l]_{\delta_1} d \ar[r] & \circ }$$ 
which is isomorphic to a wild one-relation algebra RiVIII. \medskip 

Now, it suffices to consider the case $|d^+|=1$. Suppose first that $|d^-|\geqslant 2$, and take any 
$\delta\in d^-=a^-$ different from $\alpha_1$. If $\delta\delta_1\prec I$, then by \ref{lem:3.3}, one 
gets an arrow $j\to d'=s(\delta)$, so $d'\in\{c,i\}$. We cannot have $d'=i$, since then we have double 
arrows $\alpha_1,\delta:i\to a$, which gives a wild subquiver of type $K_2^+$ after enlarging by 
$\alpha:i\to j$. Hence, we obtain an arrow $\delta:c\to a$, and moreover, every such an arrow with 
$\delta\delta_1\prec I$ must have a source in $c$. Now, observe that there is at least one arrow $\xi\in d^-$ 
different from $\alpha_1$ and $\delta$. In fact, if this was not the case, we would have 
$d^-=a^-=\{\alpha_1,\delta\}$, and hence, the following equalities:  
$$p_j=p_a^+=p_a^-=p_i+p_c=p_j^+.$$ 
This gives again a contradiction with Lemma \ref{lem:3.1}, and therefore, there is an arrow $\xi:x\to a=d$ 
different from $\alpha_1$ and $\delta$. It follows also that $\xi\delta_1\nprec I$ and we obtain the 
same wild subcategory of type RiVIII as in case $|a^\pm|\geqslant 2$ above, but with $\xi:x\to d$ instead 
of an arrow $d\to\circ$. \smallskip 

Finally, assume that $|d^-|=1$, i.e. $d$ is a $1$-regular vertex (in $Q$), and consider any $\beta'\in b^-$. 
We claim that $\beta'\beta_1\nprec I$. If this is not the case, then by \ref{lem:3.3}, we obtain an 
arrow $i\to b'=s(\beta')$, thus $b'\in\{a,j\}$. Because $b\neq c$ (and no loop at $i$), we cannot have 
$b'=j$, and hence, we get $b'=a=d$, which is a $1$-regular vertex. But then $b=j$, which is impossible. 
In a similar way, one can show that $\beta'\beta_1\alpha_1\nprec I$. If also $\beta'\beta_1\alpha\nprec I$, 
then we find a subcategory in covering given as follows 
$$\xymatrix@R=0.5cm{&& b' \ar[d]^{\beta'} && \\ c & \ar[l]_{\gamma_1} j\ar[rd]^{\beta} &   b\ar[d]^{\beta_1} && \\ 
 & a & \ar[l]_{\alpha_1} i \ar[r]^{\alpha} & j & \ar[l]_{\delta_1} d }$$ 
which is isomorphic to a wild one-relation algebra RiIV. If $\beta'\beta_1\alpha\prec I$, then by Lemma 
\ref{lem:3.6}, we conclude that $b'=c$ and $\beta'$ is an arrow $c\to b$ (we cannot have an arrow $i\to b$, 
by Theorem \ref{thm1}). But $|b^-|\geqslant 2$, as in Case 1, hence we can pick an arrow $\beta''\in b^-$ 
(different from $\beta'$) such that $\beta''\beta_1\nprec I$, $\beta''\beta_1\alpha_1\nprec I$. This time 
we have $\beta''\beta_1\alpha\nprec I$, since otherwise, using \ref{lem:3.6} again, one obtains double arrows 
$\beta',\beta'':c\to b$ satisfying $\beta'\beta_1\nprec I$ and $\beta''\beta_1\nprec I$, which gives a wild subquiver 
of type $K_2^*$. This yields a wild subcategory of the same type RiIV as above, but with $\beta'$ replaced by 
$\beta''$. \medskip 

This finishes the proof in Case 2, and consequently, we proved that $b=c$. The proof of $a=d$ follows dually. 
\end{proof} \medskip 

\normalfont 

Summing up the results obtained in Lemmas \ref{lem3}-\ref{lem6}, we end up with the following theorem. 

\begin{thm} If $i,j$ are vertices in $\bbQ$ joined by an edge in $E$, then $Q$ admits a block of one the 
forms  
$$
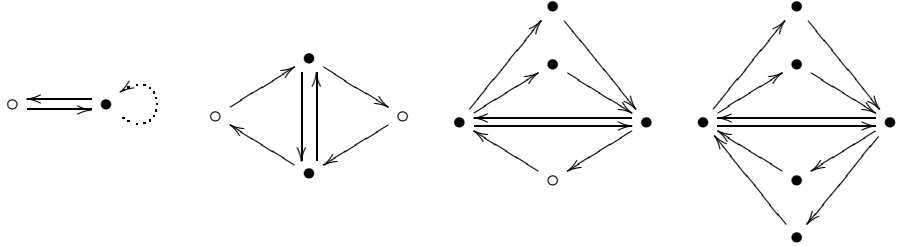
\xymatrix@R=0.4cm{\\ \\ \circ \ar@<-0.45ex>[r] & \bullet \ar@<-0.45ex>[l] \ar@(rd,ru)@{..>}[]& }
\xymatrix@R0.4cm{\\ & \bullet \ar[rd] \ar@<-0.1cm>[dd] & \\ 
\circ \ar[ru]   && \circ \ar[ld]\\ 
& \bullet \ar@<-0.1cm>[uu] \ar[lu] & } \quad 
\xymatrix@R=0.4cm{ &\bullet \ar[rdd] & \\ & \bullet\ar[rd]& \\ 
\bullet \ar[ru] \ar[ruu] \ar@<-0.35ex>[rr]&& \bullet \ar[ld] \ar@<-0.35ex>[ll]\\ & \circ \ar[lu] & } \quad 
\xymatrix@R=0.4cm{& \bullet \ar[rdd] & \\ & \bullet \ar[rd] & \\ 
\bullet \ar[ru]\ar[ruu] \ar@<-0.35ex>[rr] && \ar[ld]\ar[ldd] \ar@<-0.35ex>[ll] \bullet \\ 
& \bullet \ar[lu] & \\ & \ar[luu] \bullet & }$$ 
\end{thm} 

Note that in the first case, the dotted loop means that there may be a loop or not. In the last case, the block has no outlets, 
hence it exhausts the whole quiver $Q$. \smallskip 

The above result combined with Theorem \ref{thm1} provides the proof of The Reconstruction Theorem, 
the main result of this paper. 

\bigskip 

\section*{Acknowledgements} 
The author would like to thank Karin Erdmann for inspiring discussions and introduction to the properties 
of $4$-periodic resolutions of simple modules, which led to the discoveries presented in this paper.

\end{document}